\documentclass[11pt]{article}
\usepackage[letterpaper,textwidth=6.5in,textheight=9in,
            centering,ignorehead,nomarginpar]{geometry}

\usepackage[T1]{fontenc}
\usepackage{textcomp}
\usepackage{xcolor}

\usepackage{setspace}

\usepackage{url}
\usepackage{hyperref}
\usepackage{doi}

\usepackage{array}
\usepackage{natbib}
\bibpunct{(}{)}{;}{a}{,}{;}
\usepackage{graphicx}
\usepackage{bbm}
\usepackage{amsmath,amsthm,amssymb,amsfonts,latexsym}
\usepackage{mathtools}
\usepackage{booktabs}
\usepackage{icomma}
\usepackage{afterpage}
\usepackage{titling}
\thanksmarkseries{alph}
\usepackage[labelfont={small,sc},textfont={small,it},
            margin=15pt,skip=10pt,position=bottom]{caption}
\usepackage[labelfont={scriptsize,normalfont},textfont={scriptsize,it},
            margin=20pt,skip=5pt,position=bottom,
            labelformat=simple]{subcaption}

\usepackage{algorithmic}
\usepackage[boxed]{algorithm}

\newtheorem{remark}{Remark}
\newtheorem{assumption}{Assumption}

\newtheorem{proposition}{Proposition}
\newtheorem{corollary}{Corollary}

\numberwithin{condition}{section}
\numberwithin{assumption}{section}
\numberwithin{remark}{section}
\numberwithin{equation}{section}
\numberwithin{lemma}{section}
\numberwithin{definition}{section}
\numberwithin{theorem}{section}
\numberwithin{proposition}{section}
\numberwithin{table}{section}
\numberwithin{figure}{section}
\numberwithin{theorem}{section}
\numberwithin{corollary}{section}
\numberwithin{property}{section}
\numberwithin{algorithm}{section}

\newcommand{\EQ}{\begin{equation}}
\newcommand{\EN}{\end{equation}}
\newcommand{\EQS}{\begin{equation*}}
\newcommand{\ENS}{\end{equation*}}
\newcommand{\ds}{\displaystyle}

\newcommand*\xbar[1]{%
  \hbox{%
    \vbox{%
      \hrule height 0.5pt
      \kern0.5ex%
      \hbox{%
        \kern-0.1em%
        \ensuremath{#1}%
        \kern-0.1em%
      }%
    }%
  }%
}

\def\n1{n}
\def\argmax{\mathop{\rm arg\,max}}

\def\XL{\mathbb{W}}
\def\SX{\mathcal{X}}
\def\EWES{\text{EW-ES}_{t_0}}

\def\EWLS{\text{EW-LS}_{t_0}}

\def\EWPS{\text{EW-PS}_{t_0}}

\def\LS{\text{LS}}
\def\ES{\text{ES}}

\newcommand{\qq}{\mathfrak{q}}
\newcommand{\pp}{\mathfrak{p}}

\newcommand{\kappal}{\hat{\kappa}}
\newcommand{\kappae}{\kappa}

\newsavebox{\savepar}

\numberwithin{equation}{section}
\numberwithin{table}{section}
\numberwithin{figure}{section}

\newcommand{\myred}{\color{black}}
\newcommand{\myblue}{\color{black}}

\def\argmax{\mathop{\rm arg\,max}}

\def\W{W^{\prime}}

\usepackage[running,mathlines]{lineno}

\begin{document}
\title{
Risk Measures for DC Pension Plan Decumulation
}
\author{Peter A. Forsyth\thanks{David R. Cheriton School of Computer Science,
        University of Waterloo, Waterloo ON, Canada N2L 3G1,
        \texttt{paforsyt@uwaterloo.ca}}
  \and
        Yuying Li\thanks{David R. Cheriton School of Computer Science,
        University of Waterloo, Waterloo ON, Canada N2L 3G1,
        \texttt{yuying@uwaterloo.ca}}
}

\maketitle


\begin{abstract}
As the developed world replaces Defined Benefit (DB) pension plans
with Defined Contribution (DC) plans, there is a need to develop
decumulation strategies for DC plan holders.  Optimal decumulation
can be viewed as a problem in optimal stochastic control.  Formulation
as a control problem requires specification of an objective function,
which in turn requires a definition of reward and risk.  An intuitive
specification of reward 
is the total withdrawals over the retirement
period.  Most retirees view risk as the possibility of running
out of savings.  This paper investigates several possible
left tail risk measures, in conjunction with DC plan decumulation.
The risk measures studied include (i) expected shortfall (ii) linear
shortfall  and (iii) probability of shortfall.
We establish that,
under certain assumptions, 
the set   of optimal controls associated with all  expected reward and  expected shortfall Pareto  efficient 
frontier curves  {\myblue{is identical to}}
the set  of optimal controls for all expected reward and linear shortfall Pareto efficient frontier curves.
Optimal efficient frontiers are determined  computationally for each risk measure,
based on a parametric market model.  Robustness of these strategies
is determined by testing the strategies out-of-sample using block
bootstrapping of historical data.

\vspace{5pt}

\end{abstract}

\section{Introduction}
Internationally, there is a growing movement to replace
Defined Benefit (DB) pension plans with Defined
Contribution (DC) plans.  A study of the P7 countries\footnote{Australia, 
Canada, Japan, Netherlands, Switzerland, UK, US}
reveals that in terms of fraction of total pension assets, 
DC plans have increased
from 37\% in 2003 to 58\% in 2023 \citep{Thinking_2024}.  In terms
of individual countries, Australia has 88\% of pension assets in
DC plans, while Japan has only 5\% of pension assets in DC plans.
In the Netherlands, all DB plans will transition to 
collective DC plans by 2028.\footnote{``The End of the Dutch Defined Benefit Model A Steeper Euro Swap Curve Ahead,''
{\url{https://www.pimco.com/eu/en/insights/the-end-of-the-dutch-defined-benefit-model-a-steeper-euro-swap-curve-ahead}}}
The trend towards DC plans seems inevitable, since corporations and
governments no longer desire to take on the risk of providing the
guarantees implicit in DB plans.

During the accumulation phase of a DC plan, the burden of deciding on
an asset allocation usually is relegated to the investor.  However, upon
retirement, the DC plan holder is faced with an even bigger challenge.
During the decumulation stage of a DC plan, the retiree must decide 
on a withdrawal schedule and  an asset allocation.  
Surveys have revealed that
retirees fear running out of savings more
than death \citep{Hill_2016}.
Consequently, it seems clear that the retiree wants
to
to withdraw as much as possible, but avoid ruin.
The decumulation
problem  has been termed ``the nastiest,
hardest problem in finance,'' by William Sharpe \citep{Sharpe2017}.

While it is often suggested that retirees purchase annuities to 
reduce the risk of depletion of savings, annuities are not popular
with DC plan holders \citep{Peijnenburg2016}.  \citet{MacDonald2013}
suggest that avoidance of annuities may be entirely rational.\footnote{See
also 
{\em``When do you need insurance?''} 
{\url{https://donezra.com/217-when-do-you-need-insurance/}}
}
For example, in the North American context, true inflation
protected annuities are virtually unobtainable.

An extensive study of decumulation strategies can be found
in \citet{bernhardt-donnelly:2018}.  Some recent
strategies which involve pooling longevity risk, such as a
modern tontine \citep{Fullmer_2019_a,Weinhart_2021,Forsyth_2024} appear
promising.  However,
these types of plans are still in their infancy.

The standard wealth management advice given to retirees is 
usually some variant of the ubiquitous \emph{4\% rule} \citep{Bengen1994}.  This rule suggests that
retirees should (i) invest in a portfolio of 50\% bonds and 50\% equities,
rebalanced annually and (ii) withdraw 4\% of the initial capital each
year (adjusted for inflation).  We can consider that this advice
is given to a 65-year old retiree, who wants to be sure that he/she
does not run out of savings if he/she lives to age 95.\footnote{The probability
of a 65-year old Canadian male living to age 95 is about $0.13$.}

This advice is justified on the
basis of historical rolling 30 year periods, using US data.
A retiree following this advice would never have run out of savings
over any of these rolling thirty year periods.  Various
adjustments to this rule have been suggested many times, see
e.g. \citet{Guyton-Klinger:2006}.  However, both the advice and 
historical tests can be criticized.  Rolling thirty-year periods
obviously have very high correlations.  Use of constant weight
stock allocation is somewhat simplistic, as is use of a constant (in real
terms) withdrawal rate.  
In fact \citet{Irlam:2014} used
dynamic programming methods to conclude that deterministic (i.e.\
glide path) allocation strategies are sub-optimal.\footnote{A constant weight
strategy is trivially deterministic.}
More recently, \citet{Anarkulova_2022_a} suggest that the safe
withdrawal rate might be much lower than the the 4\% rule. In contrast
to rolling historical periods, \citet{Anarkulova_2022_a} use block bootstrap
resampling to test withdrawal strategies.  We will also use bootstrap
resampling to test our results in this paper.

Nevertheless, the {\em four per cent rule}  has seen wide adoption since the 
original publication over thirty years ago,
and can be regarded as the default advice.

Contrary to commonly held beliefs, it appears that retirees are somewhat
flexible in annual spending.  A survey in \citet{Banerjee_2021} indicates
that retirees actually adjust their lifestyle (i.e. what are perceived
as fixed expenses) to match their cash flows.   

In fact, recent surveys indicate that, if anything, many retirees
underspend on the basis of their financial assets 
\citep{Rappaport_2019,blackrock_2021}.  \citet{Browning_2016}
suggests
that these assets are being held as a reserve against unexpected medical expenses.
However, Canada has a comprehensive public health care system, yet
\citet{RePEc:sls:secfds:10}
finds that senior Canadian couples 85 and older either save or give away
about 25\% of their income.

All these facts indicate that we should allow some flexibility
in withdrawals from pension savings, in order to ameliorate sequence of return risk.

Perhaps the most rigorous approach to the decumulation problem is to
formulate this as a problem in optimal stochastic control.  The controls
in this case, are (i) the asset allocation, i.e. the stock/bond split
and (ii) the withdrawal amounts (real) per year, subject to maximum
and minimum constraints.

Of course, the first task in formulating an optimal control problem is
to specify the objective function.  One possibility is to formulate the decumulation problem in terms of
a utility function, combining the withdrawals and final portfolio value.
However, it seems clear (from the popularity of the four per cent rule),
that investors prefer to delineate the trade off between risk (running
out of savings) and reward (maximizing withdrawals).

We will consider basically the same problem as formulated by
\citet{Bengen1994}.  As a result, the obvious measure of
reward is the total of the withdrawals (inflation adjusted) over a 30 year period.
However, the choice of risk measure is not so clear.
Since retirees are primarily concerned with running out
of savings, we should be focused on left tail
measures of risk.

The objective of this paper is to carry out a thorough investigation
of the following tail risk measures, in the context of decumulation, in terms of portfolio value at
year 30:
\begin{itemize}
   \item Expected shortfall, i.e., the mean of the worst $\alpha$ fraction of the outcomes.  Typically
         $\alpha = .05$.
    \item Linear shortfall, i.e. weighting negative portfolio values linearly.
      \item Probability of final portfolio value being negative.
\end{itemize}

\textcolor{black}{We first formalize the equivalence between expected withdrawal reward and expected shortfall risk (EW-ES)
and  expected withdrawal reward and linear shortfall risk (EW-LS) 
efficient frontiers.} 
We further compare the efficient frontiers generated using all three risk measures above.
We calibrate a parametric stochastic model for stocks and bonds based on almost
a century of data.  We solve the optimal control problem via dynamic programming
using the parametric model.  The controls are tested out-of-sample, using
block bootstrap resampling of historical data \citep{politis1994,Cogneau2010,dichtl2016,
Cederburg_2022,Anarkulova_2022_a}.

One of our main results is that, under certain conditions,
the set  of optimal controls associated with all  
expected reward and  shortfall Pareto  efficient frontier curves  
{\myblue{is identical to}}
the set of optimal controls for all expected reward and linear shortfall Pareto efficient frontier curves.
Consequently the essential difference between EW-ES and EW-LS
is in the parameter which specifies tail-risk level. This
parameter is an explicit wealth level target in EW-LS  
versus a probability level in EW-ES.

We conclude that  Linear Shortfall is an excellent  practical measure of
tail risk.  Linear Shortfall (LS) is (i) trivially time consistent (ii) weights
shortfall\footnote{Being short \$100,000 is worse than being short \$1.} (iii) is close
to optimal in terms of expected shortfall and probability of shortfall
(iv) has an intuitive interpretation and (v) has robust performance in out of sample
bootstrap resampling tests.
Consequently, we recommend use of expected total withdrawals (as a measure of reward)
and linear shortfall (LS) as a measure of risk in the context of studying
decumulation strategies.

\section{Problem Setting}
Spending rules (such as the four per cent rule) are clearly popular
with retirees.  It is interesting to note the following quotation from \citep{Anarkulova_2022_a}
\begin{quote}
{\em``Current retirement spending practices demonstrate a revealed preference for spending rules over
annuitization, such that the efficacy of spending rules is an important issue. 
$\ldots$
Obtaining reliable, quantitative evidence on the 4\% rule and
alternative withdrawal rates is of critical importance given their widespread use.''}
\end{quote}

Due to its wide acceptance in wealth management, we consider the scenario
discussed in \citep{Bengen1994}.  We consider a 65-year old
retiree who desires fixed minimum annual (real) cash flows over a 30 year
time horizon.  We also impose a cap on maximum
withdrawals in any year.  From the CPM2014 table from the Canadian Institute of
Actuaries\footnote{\url{www.cia-ica.ca/docs/default-source/2014/214013e.
pdf}.}, the probability that a 65-year old Canadian male
attains the age of 95 is about 0.13.  However, use of a 30 year
time horizon is considered a prudent test for having a low probability
of running out of savings.   In addition, observe that we will not
mortality weight future cash flows, as is done when averaging over
a population for pricing annuities.  Mortality weighting does not
seem to be a useful concept for an individual retiree.

Since we allow investing in risky assets, with a minimum cash withdrawals
each year, it is possible to exhaust savings.  In this case, we continue
to withdraw cash from the portfolio, which
is equivalent to borrowing cash.  This debt accumulates at the 
borrowing rate.  Essentially, we are assuming that the investor has
other assets, e.g. real estate, which can be used as  a hedge of last
resort. In practice, accumulated debt due to exhausting savings could
be funded using a reverse mortgage, with real estate as collateral \citep{Pfeiffer_2013}.

Note that real estate is not fungible with financial assets,
except as a last resort.
This mental bucketing of assets is a common tenet of behavioral
finance \citep{Shefrin-Thaler:1988}.  As far as the 
real estate is concerned, if investments perform well,
or the retiree passes away early, then the real estate can be a bequest.

The fact that the portfolio can become negative, and the required cash
flows can add to debt, means that any tail risk measure
will  penalize these states.  Hence, the optimal
stochastic control will find strategies which make these states as
unlikely as possible.

\subsection{Notation, Formulation}
The investor has access to two funds: a stock index and a constant maturity
bond index.  At any instant in time $t$, let the {\em amount} invested in the stock index fund
be denoted by $S_t \equiv S(t)$, and similarly the amount invested in
the bond index is denoted by $B_t \equiv B(t)$.  These amounts
are real, i.e. inflation adjusted.  The total (real) value of the portfolio $W_t$ is then
\begin{eqnarray}
   W_t = S_t + B_t ~.
\end{eqnarray}
For any time dependent function $g(t)$, we use the notation
\begin{equation}
g(t^+) \equiv \displaystyle \lim_{\epsilon \rightarrow 0^+}
          g(t + \epsilon) ~~; ~~
 g(t^-) \equiv \displaystyle  \lim_{\epsilon \rightarrow 0^+}
          g(t - \epsilon)  ~~.
\end{equation}
Consider a set of discrete withdrawal/rebalancing times $\mathcal{T}$,
\begin{equation}
  \mathcal{T} = \{t_0=0 <t_1 <t_2< \ldots <t_M=T\}  \label{T_def},
\end{equation}
where $T$ is the investment horizon.
For ease of notation, we assume that $t_i - t_{i-1} = \Delta t =T/M$ is constant.

At each rebalancing time $t_i, i=0,\ldots,M-1$, the investor first (i) withdraws an
amount of cash $\qq_i$ from the portfolio and then (ii) rebalances the
portfolio.  More precisely
\begin{eqnarray}
  W(t_i^+) & = & W(t_i^-) - \qq_i ~.
\end{eqnarray}

{\myblue{
Denote the state of the system at each time by $\SX(t), t \in [0,T]$.  Informally,
the state can be regarded as the information necessary to model the system
from time $t$ onwards \citep{Powell_2025_a}.
}
}

Let  the rebalancing control $\pp(\SX(t_i^-))$ be the fraction in stocks after
withdrawals, then,
\begin{eqnarray}
  S(t_i^+) & = & \pp_i( \SX(t_i^-)) W(t_i^+) \nonumber \\
          & & \pp_i( \SX(t_i^-)) \equiv \pp(  \SX(t_i^-), t_i) 
                 \nonumber \\
  B(t_i^+ ) & = & W(t_i^+) - S(t_i^+) ~.
            \label{p_control_def}
\end{eqnarray}
We can regard the amount withdrawn $\qq_i(\cdot)$ as an additional control i.e.
$\qq_i( \SX(t_i^-)) = \qq(\SX(t_i^-), t_i)$. 
{\myblue{Note we make the implicit assumption that the optimal controls are of
feedback form, i.e. only  a function of the state and time.}}

{\myblue{
Based on the parametric SDE model for $(S_t,B_t)$ in Appendix \ref{parametric_model_appendix}
and \citet{forsyth:2022}, we will assume in the following that
$\SX(t) = (S(t), B(t)), t \in [0,T]$, with the
realized state of the system denoted by $x = (s,b)$.
More generally, of course, it may be necessary to include other variables
to define the state (e.g. {\em lifting the state space} to include
path dependent variables).\footnote{A classic example is the pricing of an Asian option,
which depends of the observed average stock price $A_t$.  If  the stock price $S_t$ follows GBM,
then the state space for an Asian option is lifted to  $(S_t, A_t)$.}
}
}

In the special case that there are no transaction costs $\qq_i(\cdot) = \qq_i(W_i^-)$ and
$\pp_i(\cdot) = \pp(W_i^+)$, i.e. the amount withdrawn is only a function of
total wealth before withdrawals, and the rebalancing fraction is only
a function of wealth after withdrawals.  Note that it is straightforward to
include transaction costs, but if typical costs for ETFs are included,
this has a very small impact on the controls \citep{dang-forsyth:2014a}.

The control at time $t_i$ is given by $(\qq_i(\cdot), \pp_i(\cdot) )$, where
$(\cdot)$ denotes the control as a function of state.  We specify feasibility of control by prescribing the 
set of admissible {\em values} of the controls by $\mathcal{Z}$, i.e.,
\begin{equation}
  (\qq_i,\pp_i) \in \mathcal{Z}(W_i^-, W_i^+,t_i) = 
    \mathcal{Z}_{\qq}(W_i^-, t_i) \times \mathcal{Z}_{\pp} (W_i^+,t_i)~.
\label{admiss_set}
\end{equation}
where
\begin{linenomath}
\begin{align}
  \mathcal{Z}_{\qq}(W_i^-, t_i) &=
    \begin{cases}
      [\qq_{\min},\qq_{\max}] & t_i \in \mathcal{T} ~;~ t_i \neq t_M~;~
        W_i^- \geq \qq_{\max} \\
      [\qq_{\min},\max(\qq_{\min},W_i^-)] & t_i \in \mathcal{T} ~;~
        t \neq t_M ~;~ W_i^- < \qq_{\max} \\
      \{0\} & t_i = t_M
   \end{cases} ~, 
\label{Z_q_def} \\
  \mathcal{Z}_\pp (W_i^+,t_i) &=
    \begin{cases}
      [0,1] & W_i^+ > 0 ~;~ t_i \in \mathcal{T}~;~ t_i \neq t_M \\
      \{0\} & W_i^+ \leq 0 ~;~ t_i \in \mathcal{T}~;~  t_i \neq t_M \\
      \{0\} & t_i=t_M
    \end{cases} ~.
\label{Z_p_def}
\end{align}
\end{linenomath}

These expressions encapsulate the following constraints:
\begin{itemize}
   \item No shorting, no leverage (assuming solvency, i.e., when $W_i^+ > 0$),
   \item Maximum $\qq_{\max}$ and minimum $\qq_{\min}$ withdrawal constraints,
   \item In the case of insolvency $W_i^+ < 0$, trading ceases and debt accumulates
         at the borrowing rate,
   \item At $t=t_M$, all stocks are liquidated no withdrawals $\qq_M = 0$,
   \item If $W_i^- < \qq_{\max}$, the investor attempts to avoid insolvency, 
          but always withdraws at least $\qq_{\min}$.  
\end{itemize}

Recall that we assume that the retiree can finance the debt using 
other assets, e.g. a real estate hedge of last resort.  At first sight
it might seem appropriate to simply cease withdrawals if insolvent.  However,
by assumption, the retiree needs a minimum cash flow of $\qq_{\min}$ each
year.  Therefore, we penalize any set of controls which 
causes the retiree to exhaust his savings (and access the
assumed real estate hedge) in order to fund the minimum
cash flows.  Allowing debt to accumulate also penalizes early insolvency
compared to late insolvency.

The admissible control set $\mathcal{A}$ can then be written as
\begin{equation}
  \mathcal{A} = \biggl\{
    (\qq_i, \pp_i)_{0 \leq i \leq M} : (\pp_i, \qq_i) \in 
    \mathcal{Z}(W_i^-, W_i^+,t_i) \biggr\}~.
\end{equation}
For notational simplicity, we denote a dynamic control by  $\mathcal{P}$,  and
an admissible control $\mathcal{P} \in \mathcal{A}$ can be written as
\begin{equation}
  \mathcal{P} = \{(\qq_i(\cdot), \pp_i(\cdot)) ~:~ i=0, \ldots, M \} ~.
\end{equation}
We also define $\mathcal{P}_n \equiv \mathcal{P}_{t_n} \subset
\mathcal{P}$ as the tail of the set of controls in $[t_n, t_{n+1},
\ldots, t_{M}]$, i.e.\
\begin{equation}
  \mathcal{P}_n =\{(\qq_n(\cdot), \pp_n(\cdot)), \ldots, 
                   (\qq_{M}(\cdot), \pp_{M}(\cdot)) \} ~.
\end{equation}
For notational completeness, we also define the tail of the admissible
control set $\mathcal{A}_n$ as
\begin{equation}
  \mathcal{A}_n = \biggl\{
    (\qq_i, \pp_i)_{n \leq i \leq M} : 
    (\qq_i, \pp_i) \in \mathcal{Z}(W_i^-, W_i^+,t_i) \biggr\}~,
\end{equation}
so that $\mathcal{P}_n \in \mathcal{A}_n$.

\section{Risk and reward}
\subsection{Reward}
Define $E_{\mathcal{P}_0}^{\SX_0^-,t_0^-} [ \cdot ]$
as the expectation  conditional on the observation at time $t_0^-$, state $\SX_0^-$, under
control $\mathcal{P}_{0}$.  We then define reward as
\begin{eqnarray}
    {\text{ EW }} (\SX_0^-, t_0^-)  & = &  E_{\mathcal{P}_{0}}^{\SX_0^-,t_{0}^-} 
                \biggl[ 
                    \sum_{i=0}^{M} \qq_i 
                 \biggr]
\end{eqnarray}
which is the total expected withdrawals in $[0,T]$.
We will use EW as the reward measure in all cases.
Note that $\qq_i$ is inflation adjusted and that we do not discount the future cash flows.  We view this
as a conservative approach and is consistent with the \citet{Bengen1994} scenario.

\subsection{Risk}

\begin{description}

 \item[PS] We define PS risk as the probability of shortfall w.r.t. a terminal wealth level $\XL$,
          \begin{eqnarray}
  {\text{ PS}} (\SX_0^-, t_0^-)   = Prob[ W_T < 
     \XL] = E_{\mathcal{P}_0}^{\SX_0^-,t_0^-}[ {\bf{1}}_{ W_T < \XL } ]
      ~.
     \end{eqnarray}
        Usually, $\XL$ is zero,
        i.e., we are concerned with running out of cash.
         We want to {\em{minimize}} PS risk.

  \item[LS] Linear shortfall
      \begin{eqnarray}
       {\text{LS}}(\SX_0^-, t_0^-) & = & 
             E_{\mathcal{P}_{0}}^{\SX_0^-,t_{0}^-}[ \min (W_T -\XL, 0) ]~.
       \end{eqnarray}
     Note that PS risk does not differentiate bad outcomes.  Clearly,
     being short $1\$ $ is not as bad as being short $ 1000\$ $.
     LS weights the bad outcomes.
     Since ES is defined in terms of final wealth,
    not losses, we want to {\em maximize} LS risk measure.

 \item[ES]
   ES  is the mean of the worst $\alpha$ fraction of
            outcomes.  A common choice is $\alpha = .05$.
             More precisely, let $W_T$ be the wealth associated with ${\mathcal{P}_{0}}^{\SX_0^-,t_{0}^-} $
             \begin{align}
          {\text{ES}}(\SX_0^-, t_0^-) & = E_{\mathcal{P}_{0}}^{\SX_0^-,t_{0}^-} 
                       \biggl[ \frac{ W_T  {\bf{1}}_{W_T < \XL} }{\alpha} \biggr]  
                         \nonumber \\
             & \text{ subject to } 
                    \begin{cases}
                        E_{\mathcal{P}_{0}}^{\SX_0^-,t_{0}^-} [ {\bf{1}}_{W_T < \XL}]  =  \alpha.  \label{def:ES} \\
                    \end{cases}
             \end{align}
            We want to {\em maximize} ES  risk measure.

\end{description}
One of the main goals of this paper is to compare and contrast these different reward-risk combinations, both mathematically and computationally.

\subsection{Summary of Acronyms}
For future reference, Table \ref{acro_def} lists
the acronyms  used in this paper.

\begin{table}[tb]
\begin{center}
{\small
\begin{tabular}{lc} \toprule
Acronym          & Description \\ \midrule 
 EW (expected withdrawals)   &  $E[ \sum_{i=0}^{M} \qq_i]$\\
 PS (probability of shortfall) &  $E[{\bf{1}}_{ W_T < \XL }]$\\
 LS (linear shortfall) &       $E[\min (W_T -\XL, 0)]$\\
 ES (expected shortfall )      & $ E \biggl[ \frac{{ W_T \bf{1}}_{W_T < \XL} }{\alpha} \biggr]$\\
                   & s.t. $E[{\bf{1}}_{W_T < \XL}] = \alpha $\\
\bottomrule
\end{tabular}
}
\end{center}
\caption{Definition of acronyms.
\label{acro_def}
}
\end{table}

\section{Pareto points} 
We will use a scalarization technique to determine Pareto optimal
points for the multi-objective problems balancing risk and reward.
As an example consider problem EW-PS.  Informally, given an scalarization parameter $\kappa  >0$, we seek the
optimal control $\mathcal{P}_0$ that maximizes
\begin{eqnarray}
      \text{EW}( \SX_0^-, t_0^-) - \kappa~\text{PS}( \SX_0^-, t_0^-)~.
\end{eqnarray}
Varying $\kappa$ traces out an efficient frontier in the (EW, PS) plane.  For any fixed
value of PS, the corresponding point on the efficient frontier is the largest possible
value of EW.

\subsection{PS, LS}\label{PS_QS_LS_prob}
We solve optimal control problem for weighted reward and risk combinations,  e.g., EW-PS, EW-LS.
 To be precise, for each reward and risk parameter pair, we define the function $G(W_T, \XL)$ below,
\begin{eqnarray}
     \text{PS}&: &  G_{PS}(W_T, \XL)  =   - {\bf{1}}_{ W_T < \XL } \\
      \text{LS} & : &  G_{LS}(W_T, \XL)  =  \min (W_T -\XL, 0) ~,
\end{eqnarray}
where $\XL$ is a specified wealth level. 
Assuming a risk aversion scaling parameter $\kappa$,  the general problem for EW-xS, ($x = \{\text{P,L}\}$)  can be written as
\begin{align}
{\text{EW-xS}}_{t_0}\left(\XL, \kappa \right):
  \: 
  & \sup_{\mathcal{P}_{0}\in\mathcal{A}} 
     \Biggl\{E_{\mathcal{P}_{0}}^{\SX_0^+,t_{0}^+}
      \Biggl[
         \sum_{i=0}^{M} \qq_i +
     \kappa G_{\text{xS}}(W_T, \XL)  
     \Biggr. \Biggr. \nonumber \\
  &  \Biggl. \biggl.  \quad +~\epsilon W_T
     \bigg\vert \SX(t_0^-) = (s,b) ~\Biggr] \Biggr\}\label{PCES_a}\\
\text{ subject to } &
\begin{cases}
  (S_t, B_t) \text{ follow processes \eqref{jump_process_stock} and
                    \eqref{jump_process_bond}};
                     ~~t \notin \mathcal{T} \\
  W_{\ell}^+ = W_{\ell}^{-} - \qq_\ell \,;
              ~\SX_\ell^+ = (S_\ell^+ , B_\ell^+) \\
  W_{\ell}^- = S(t^-_i) + B(t_i^-) 
               \\
  S_\ell^+ = \pp_\ell(\cdot) W_\ell^+ \,;
             ~B_\ell^+ = (1 - \pp_\ell(\cdot) ) W_\ell^+ \, \\
  (\qq_\ell(\cdot) , \pp_\ell(\cdot)) \in
    \mathcal{Z}(W_\ell^-, W_\ell^+,t_\ell) \\
  \ell = 0, \ldots, M ~;~ t_\ell \in \mathcal{T} \\
\end{cases}~.
\label{PCES_b}
\end{align}

Observe that we have added the stabilization term $\epsilon W_T$ to the objective
function in equation (\ref{PCES_a}).  The control problem is ill-posed in
the cases where $t \rightarrow T, W_t \gg \XL$.  In this case, due to the maximum
withdrawal constraint, and since the $Prob[W_t < \XL] \simeq 0$, then the control
has almost no effect on the objective function.  Addition of the stabilization
term regularizes the problem (see e.g. 
\citet{Chen_Mohib_2023}).  We will discuss this further in later sections.

\subsection{Expected Shortfall (ES)}\label{ES_prob}
We are interested in the relationship between the
above reward-risk formulations with the same reward but ES risk, i.e.,
\begin{eqnarray}
   \text{EW}( \SX_0^-, t_0^-) + \kappa~\text{ES}(\SX_0^-, t_0^-)~.
\end{eqnarray}
 We formulate  the EW-ES optimal control problem using the technique 
in \citet{Uryasev_2000}, more precisely ($0 <\alpha <1$)
\begin{align}
{\text{EW-ES}}_{t_0}\left(\alpha, \kappa \right):
  \:
        &    \sup_{\mathcal{P}_{0}\in\mathcal{A}}
                \Biggl\{
                 E_{\mathcal{P}_{0}}^{\SX_0^-,t_{0}^-}
             \Biggl[ ~\sum_{i=0}^{M} q_i ~  + ~
                \kappa \sup_{\W}\biggl( \W + \frac{1}{\alpha} \min (W_T -\W, 0) \biggr) 
                \Biggr. \Biggr. \nonumber \\
        &  ~~~~ \Biggl. \Biggl.
               + \epsilon W_T
                  \bigg\vert  (s,b)
                     ~\Biggr] \Biggr\} ~ \label{EW_ES_0} \\
\text{ subject to } &
\begin{cases}
   {\text{Conditions~}} (\ref{PCES_b})
\end{cases}~.
    \nonumber
\end{align}

Interchanging the order in $sup~sup\{\cdot\} $ in problem (\ref{EW_ES_0} ), we equivalently have
\begin{align}
{\text{EW-ES}}_{t_0}\left(\alpha, \kappa \right):
  \:         & \sup_{\W}   \sup_{\mathcal{P}_{0}\in\mathcal{A}}
                \Biggl\{
                 E_{\mathcal{P}_{0}}^{\SX_0^-,t_{0}^-}
             \Biggl[ ~\sum_{i=0}^{M} q_i ~  + ~
                \kappa \biggl( \W + \frac{1}{\alpha} \min (W_T -\W, 0) \biggr)
                \Biggr. \Biggr. \nonumber \\
        &  ~~~~ \Biggl. \Biggl.
               + \epsilon W_T
                  \bigg\vert \SX_0^-= (s,b)
                     ~\Biggr] \Biggr\} ~ \label{EW_ES_1} \\
\text{ subject to } &
\begin{cases}
   {\text{Conditions~}} (\ref{PCES_b})
\end{cases}~.
    \nonumber
\end{align}
Note that, as for the EW-xS problems, we have added a 
stabilization term $\epsilon W_T$ to the objective function.

\begin{remark}[Pre-commitment policy]
Note that the optimal control for problem (\ref{EW_ES_1}) is formally
a pre-commitment policy \citep{forsyth_2019_c}.  We delay further
discussion concerning this issue to Section \ref{subsec:EWEStoEWLS}.
\end{remark}

\subsection{Properties of optimal solution of EW-ES formulation  (\ref{EW_ES_1})}

 Let $\mathcal{P}_0$ be any permissible control for problem (\ref{EW_ES_1}) and 
 $W_T$ be the wealth corresponding to $\mathcal{P}_0$.
Consider the maximizer below:\footnote{The $\argmax$ is well defined since
$\sup_{\mathcal{P}} \{ \cdot \}$ is
a continuous function of $W^{\prime}$.}

\begin{eqnarray}
    \XL & = & \argmax_{W^{\prime}}
                \Biggl\{
                 E_{\mathcal{P}_{0}}^{\SX_0^-,t_{0}^-}
             \Biggl[ ~\sum_{i=0}^{M} q_i ~  + ~
                \kappa \biggl( \W + \frac{1}{\alpha} \min (W_T -\W, 0) \biggr) 
                \Biggr. \Biggr. \nonumber \\
        & & ~~~~ \Biggl. \Biggl.
               + \epsilon W_T
                  \bigg\vert  \SX_0^-=(s,b)
                     ~\Biggr] \Biggr\} ~ .
           \label{W_star_00}
\end{eqnarray}
Following  \cite{Uryasev_2000},  it can be shown that \eqref{W_star_00}  is equivalent to the probability constraint below, under the assumption of continuity in distribution of $W_T$,
\begin{eqnarray}
      E_{\mathcal{P}_{0}}^{\SX_0^-,t_{0}^-} [ {\bf{1}}_{W_T< \XL}] & = & \alpha  \label{PCEE_props_1}
      ~.
\end{eqnarray}
Let $E_{\mathcal{P}_{0}}$ denote $E_{\mathcal{P}_{0}}^{\SX_0^-,t_{0}^-}$ for notational simplicity. Consider
\begin{eqnarray}
     & & E_{\mathcal{P}_{0}}  \biggl( \XL + \frac{1}{\alpha} \min (W_T -\XL, 0) \biggr) \nonumber \\
   & & {\text{subject to }} \nonumber \\
   & & ~~~~~~~~~ \begin{cases}
                   E_{\mathcal{P}_{0}}[ {\bf{1}}_{W_T \leq \XL}] = \alpha
                 \end{cases} 
               ~.
                \label{check_1} 
\end{eqnarray}
Let $g_{\mathcal{P}_{0}}(W_T)$ be the density of $W_T$ under control $\mathcal{P}_{0}$.  Then,
write equation (\ref{check_1}) as
\begin{eqnarray}
    & & \int_{-\infty}^{+\infty} \XL~ g_{\mathcal{P}_{0}}(W_T) ~dW_T
     + \frac{1}{\alpha} \int_{-\infty}^{\XL}  (W_T - \XL)~g_{\mathcal{P}_{0}}(W_T) ~dW_T \label{check_3}
               \\
    & &  \text{subject to } \nonumber \\
   & & ~~~~~~~~~ \begin{cases}
                   \int_{-\infty}^{\XL} g_{\mathcal{P}_{0}}(W_T) ~dW_T = \alpha
                \end{cases}  ~ . \label{check_4}
\end{eqnarray}
We can write (\ref{check_3}) as
\begin{eqnarray}
       \XL ~\int_{-\infty}^{+\infty}~ g_{\mathcal{P}_{0}}(W_T) ~dW_T
          -\frac{\XL}{\alpha} \int_{-\infty}^{\XL} G_{\mathcal{P}_{0}}(W_T) ~dW_T 
        + \frac{1}{\alpha} \int_{-\infty}^{\XL} W_T~ g_{\mathcal{P}_{0}}(W_T) ~dW_T
        ~.
               \label{check_5}
\end{eqnarray}
Using equation (\ref{check_4}), this becomes
\begin{eqnarray}
    \XL - \XL + \frac{1}{\alpha} \int_{-\infty}^{\XL} W_T~ g_{\mathcal{P}_{0}}(W_T) ~dW_T 
    & = & E_{\mathcal{P}_{0}}^{\SX_0^-,t_{0}^-} \biggl[ \frac{{W_T \bf{1}}_{W_T < \XL }}{\alpha} \biggr]
  ~.
\end{eqnarray}
Thus,  when \eqref{W_star_00}  is satisfied,  we have
\begin{eqnarray}
       E_{\mathcal{P}_{0}}^{\SX_0^-,t_{0}^-} [ {\bf{1}}_{W_T< \XL}] & = & \alpha  \nonumber \\
      E_{\mathcal{P}_{0}}^{\SX_0^-,t_{0}^-} \biggl[ \frac{{W_T \bf{1}}_{W_T < \XL }}{\alpha} \biggr] & = & 
               E_{\mathcal{P}_{0}}^{\SX_0^-,t_{0}^-} \biggl[ \biggl( \XL + \frac{1}{\alpha} 
                             \min (W_T -\XL, 0) \biggr) \biggr] 
           ~.   \label{equiv_1}
    \end{eqnarray}
Consider the optimal $\XL^*$ and control  $\mathcal{P}_0^*$  from $\EWES(\alpha,\kappa)$,
equation (\ref{EW_ES_1}), i.e.,
\begin{eqnarray}
    \XL^* & = & \argmax_{W^{\prime}}
      \sup_{\mathcal{P}_{0}\in\mathcal{A}}
                \Biggl\{
                 E_{\mathcal{P}_{0}}^{\SX_0^-,t_{0}^-}
             \Biggl[ ~\sum_{i=0}^{M} q_i ~  + ~
                \kappa \biggl( \W + \frac{1}{\alpha} \min (W_T -\W, 0) \biggr)
                \Biggr. \Biggr. \nonumber \\
        &   &~~~~ \Biggl. \Biggl.
               + \epsilon W_T
                  \bigg\vert \SX_0^-= (s,b)
                     ~\Biggr] \Biggr\}  ~,
           \label{W_star_1}
\end{eqnarray}
then equation (\ref{equiv_1}) implies
\begin{eqnarray}
    Prob[ W_T^* < \XL^*] & = & \alpha   \label{PCEE_props_2} \\
    ES &= &{\text{ mean of worst $\alpha$ fraction of outcomes}}  \nonumber \\
       & = & 
               E_{\mathcal{P}^*_{0}}^{\SX_0^-,t_{0}^-} \biggl[ \biggl( \XL^* + \frac{1}{\alpha} 
                             \min (W_T^* -\XL^*, 0) \biggr) \biggr] ~. \nonumber
\end{eqnarray}
From \eqref{PCEE_props_2},  we see immediately that $\XL^* $ is the $\alpha$-VaR (value at risk) of the terminal wealth $W_T^*$ associated with the optimal control.

Fixing any target wealth level $\XL$, we  consider linear 
shortfall Pareto optimization ($\hat{\kappa} >0$):
\begin{eqnarray}
\text{EW-LS}_{t_0}\left(\XL, \widehat{\kappa} \right):
    \qquad &&
     \sup_{\mathcal{P}_{0}\in\mathcal{A}}
        \Biggl\{
               E_{\mathcal{P}_{0}}^{\SX_0^-,t_{0}^-}
           \Biggl[ ~\sum_{i=0}^{M} q_i ~  + ~
               \hat{\kappa} ~\min (W_T - \XL, 0)
             + \epsilon W_T
                    \Biggr. \Biggr. \nonumber \\
         & &~~~~~~~~~~~~~~~~~~~~~~  \Biggl. \Biggl. ~~~~~
                \bigg\vert \SX(t_0^-) = (s,b)
                   ~\Biggr] \Biggr\} ~,~\label{EWLS_a}\\
                   \text{ subject to } &&
\begin{cases}
   {\text{Conditions~}} (\ref{PCES_b})
\end{cases}~.
    \nonumber
\end{eqnarray}
Note that we notationally distinguish  risk aversion parameters for EW-ES and EW-LS to describe their precise connection.
We summarize the relationship
between EW-ES and EW-LS in Proposition \ref{prop_time} (see also \citet{forsyth_2019_c}).

\begin{proposition}[Optimal EW-ES strategy solves EW-LS] \label{prop_time} 

\begin{itemize}
  \item[] 

 \item[(i)] The pre-commitment  strategy  $\mathcal{P}^*$  which solves $\EWES(\alpha,{\kappa})$
 \eqref{EW_ES_0} is a solution to $\EWLS (\XL,\frac{\kappa}{\alpha})$ \eqref{EWLS_a} with the fixed  wealth level
$\XL=\XL^*$ defined in \eqref{W_star_1}. 

\item[(ii)] Conversely, 
an optimal control for $\EWLS (\XL,\frac{\kappa}{\alpha})$ \eqref{EWLS_a}
with the fixed  wealth level
$\XL=\XL^*$ given by \eqref{W_star_1}, solves  $\EWES(\alpha,{\kappa})$ \eqref{EW_ES_0}.
\end{itemize}

\end{proposition}

\begin{proof} 
\begin{itemize}
 \item[]

\item[(i)]
Let  $\mathcal{P}_0^*$ solve (\ref{EW_ES_0}). Then it solves \eqref{EW_ES_1} due to equivalence between  (\ref{EW_ES_1}) and  \eqref{EW_ES_0}.
Consequently $\mathcal{P}_0^*$ also  solves the linear shortfall problem
$\EWLS (\XL^*,\frac{\kappa}{\alpha})$ 
in \eqref{EWLS_a},  
i.e., $\mathcal{P}_0^*$  solves
\begin{align}
{\text{EW-LS}}_{t_0}\left(\XL, \kappal \right):
  \:      \quad \quad   &   \sup_{\mathcal{P}_{0}\in\mathcal{A}}
                \Biggl\{
                 E_{\mathcal{P}_{0}}^{\SX_0^-,t_{0}^-}
             \Biggl[ ~\sum_{i=0}^{M} q_i ~  + ~
              \kappal  \min (W_T -\XL ,0) 
                \Biggr. \Biggr. \nonumber \\
        &  ~~~~ \Biggl. \Biggl.
               + \epsilon W_T
                  \bigg\vert  (s,b)
                     ~\Biggr] \Biggr\} ~ \nonumber\\
\text{ subject to } &
\begin{cases}
   {\text{Conditions~}} (\ref{PCES_b})
\end{cases}~,
    \nonumber
\end{align}
with $\kappal= \frac{\kappa}{\alpha}$,
$\XL=\XL^*$,  and $\XL^*$ defined in \eqref{W_star_1}.

\item[(ii)] Assume  that $\mathcal{P}_0^*$  solves  $\EWLS(\XL^*, \kappal)$, \eqref{EWLS_a}, with $\kappal= \frac{\kappa}{\alpha}$ and $\XL^*$ defined in \eqref{W_star_1}. 
Then  $\mathcal{P}_0^*$  solves
\begin{align}
   \quad \quad   &   \sup_{\mathcal{P}_{0}\in\mathcal{A}}
                \Biggl\{
                 E_{\mathcal{P}_{0}}^{\SX_0^-,t_{0}^-}
             \Biggl[ ~\sum_{i=0}^{M} q_i ~  + ~
              \kappa ( \XL^* + \frac{1}{\alpha}  \min (W_T -\XL^* ,0) 
                \Biggr. \Biggr. \nonumber \\
        &  ~~~~ \Biggl. \Biggl.
               + \epsilon W_T
                  \bigg\vert  (s,b)
                     ~\Biggr] \Biggr\} ~ \nonumber\\
\text{ subject to } &
\begin{cases}
   {\text{Conditions~}} (\ref{PCES_b})
\end{cases}~,
    \nonumber
\end{align}
Applying $\XL^*$ defined in \eqref{W_star_1},   then $\mathcal{P}_0^*$  solves \eqref{EW_ES_1}, and hence \eqref{EW_ES_0}.
\end{itemize}
\end{proof}
{\myblue{
Let
\begin{eqnarray}
  \mathcal{D}_{ES} & = &  \{ (\alpha , \kappa) ~|~ 0 < \alpha < 1 ~;~ \kappa > 0 \}
                 \nonumber \\
  \mathcal{D}_{LS} &=& \{ (\mathbb{W}, \kappa) ~|~ \mathbb{W} \in \mathbb{R} ~;~ \kappa > 0 \} \label{DESLS}
         ~.
\end{eqnarray}
Define
\begin{equation}\label{optSet}
\begin{array}{l}
\mathcal{H}_\ES^*=\{  \mathcal{P}_0^*: ~\mathcal{P}_0^*~\text{solves} ~
        \EWES(\alpha, \kappa) \eqref{EW_ES_0} \text{ for some }
            (\alpha,\kappa)  \in \mathcal{D}_{ES} \}\\
\mathcal{H}_\LS^*=\{  \mathcal{P}_0^*:~\mathcal{P}_0^*~\text{solves} ~
    \EWLS(\XL, \kappal)  \eqref{EWLS_a} \text{ for some }
            (\XL,\kappal)  \in \mathcal{D}_{LS} \} ~.
\end{array}
\end{equation}
}
}
We then have the following Corollary, which follows from Proposition \ref{prop_time}:
\begin{corollary}\label{corollary_1} Let  $\mathcal{H}_{\ES}^*$ 
and $\mathcal{H}^*_{\LS}$ be defined in \eqref{optSet}.
Then the set  $\mathcal{H}_{\ES}^*$ of optimal controls 
for Problem $\EWES$
is a subset of the set   $\mathcal{H}^*_{\LS}$ 
of optimal controls for Problem $\EWLS$.
\end{corollary}

\subsection{Time inconsistent EW-ES and time consistent EW-LS}\label{subsec:EWEStoEWLS}

While Proposition \ref{prop_time} indicates that $\EWES$ and $\EWLS$ share a 
common solution when the wealth level $\XL=\XL^*$, 
these two dynamic optimization formulations have 
different properties in terms of time consistency.
To see this,  we first recall  the concept of time consistency and relate its relevance 
to the $\EWES(\alpha,\kappa)$ problem, (\ref{EW_ES_1}).

Consider the optimal control $\mathcal{P}_0^*=  (\mathcal{P}^*)^{t_0}$ computed at $t_0$ from problem (\ref{EW_ES_1}) at all rebalancing times, 

\begin{eqnarray} \label{opt_P_t0}
    (\mathcal{P}^*)^{t_0}(\SX(t_i^-), t_i)  ~, ~ i=0, \ldots, M ~,
\end{eqnarray}
i.e., \eqref{opt_P_t0} denotes the optimal control  $(\mathcal{P}^*)^{t_0}$  at any time $t_i \geq t_0$, as 
a function of the state variables $\SX(t)$.

Next we solve the problem (\ref{EW_ES_1})  starting 
at a later time $t_k, k > 0$ and denote the  optimal control starting at $t_k$ is denoted by:
\begin{eqnarray} 
    (\mathcal{P}^*)^{t_k}(\SX(t_i^-), t_i)  ~, ~ i=k, \ldots, M  ~.
\end{eqnarray}
In general, the solution of (\ref{EW_ES_1}) 
computed at $t_k$ is not equivalent to the solution computed at $t_0$:
\begin{eqnarray} 
    (\mathcal{P}^*)^{t_k}(\SX(t_i^-), t_i)  
       \neq  (\mathcal{P}^*)^{t_0}(\SX(t_i^-), t_i)  ~;~ i \ge k > 0.
\end{eqnarray}
This non-equivalence makes problem (\ref{EW_ES_1}) \emph{time inconsistent}, implying that
 the optimal control computed at  $t_k$, $k>0$, deviates from the control determined at time $t_0$. 
 The optimal control  $\mathcal{P}_0^*=  (\mathcal{P}^*)^{t_0}$ determined at the initial time 
is considered a {\em pre-commitment}
control since the investor would need 
to commit to following the strategy at all times following $t_0$, 
even if the optimal control recomputed at future time becomes different.
Some authors describe pre-commitment controls as non-implementable because of the incentive to deviate
from the initial control.

Following Proposition \ref{prop_time},  the pre-commitment control for $\EWES$ (\ref{EW_ES_1}),  
fortunately,  can be shown
to be optimal for  $\EWLS(\XL,\kappal)$, for which $\XL$ is fixed at the
optimal value  (at time zero) in \eqref{W_star_1}.

With a fixed $\XL$, ${\text{EW-LS}}_{t_0}(\XL,\kappa / \alpha)$
uses a target-based linear shortfall as its measure of risk,
and $\EWLS$  is trivially time consistent. 
Furthermore,
$\XL$ has the convenient interpretation of a disaster level of final wealth,
as specified at time zero.

While the pre-commitment strategy $\mathcal{P}^*$ from $\EWES(\alpha,{\kappa})$,  (\ref{EW_ES_1}), is time inconsistent when viewed as a solution to  EW-ES, this strategy is time consistent with respect  to 
 $\EWLS(\XL^*,\frac{\kappa}{\alpha}$) with the fixed wealth level $\XL^*$.  
In other words,  conditional on information at $t_n$ and fixed $\XL^*$, 
the future decision  $\{(\mathcal{P}^*)^{t_n}(\SX(t_i^-), t_i)  ~;~ i=n, \ldots, M \} $  
 of the
  optimal  pre-commitment EW-ES  control, computed at $t_0$, 
solves
\begin{eqnarray}
 \text{EW-LS}_{t_n}\left(\XL^*,\kappa / \alpha  \right):
    &  &
            \sup_{\mathcal{P}_{n}\in\mathcal{A}}
        \Biggl\{
               E_{\mathcal{P}_{n}}^{\SX_n^-,t_{n}^-}
           \Biggl[ ~\sum_{i=n}^{M} q_i ~  + ~
               \frac{\kappa}{\alpha} \min (W_T - \XL^*,0) 
                    \Biggr. \Biggr.   \label{timec_equiv}
\\
        & & ~~~~ \Biggl. \Biggl.
               + \epsilon W_T
                \bigg\vert \SX(t_n^-)= (s,b)
                   ~\Biggr] \Biggr\},
                   \nonumber 
                  \end{eqnarray}
for any given permissible stock and bond value pair $(s,b)$.

\begin{remark}[EW-ES $\rightarrow$ EW-LS]
Proposition \ref{prop_time} essentially tells us that any optimal control
$\mathcal{P}^*$ from EW-ES problem (\ref{EW_ES_1}), solves
some EW-LS problem (\ref{timec_equiv}) with a fixed wealth level $\XL^*$.
Since  EW-LS is time consistent,  the 
EW-ES optimal control  $\mathcal{P}^*$ is time consistent when 
Pareto optimality is assessed with EW-LS with this fixed wealth level $\XL^*$.
\end{remark}

Since the optimal control   $\mathcal{P}^*$  for $\text{EW-ES}_{t_0}(\alpha, \kappa)$  
solves $\text{EW-LS}_{t_n}(\XL^*, \kappa / \alpha)$ at any $t_n$, where $\XL^*$ is the 
$\alpha$-VaR of the conditional terminal
wealth $W_T^*$, conditional on $W_0^*=s+b$, we  can regard  $\mathcal{P}^*$ as an induced time consistent strategy for $\text{EW-LS}_{t_n}(\XL^*, \kappa / \alpha)$
 \citep{Strub_2019_a}. Consequently
the investor has no incentive to deviate from the induced time
consistent strategy, determined at time zero.  Hence this policy is implementable.

For more detailed analysis concerning the subtle distinctions involved in
pre-commitment, time consistent, and induced
time consistent strategies, please consult
\cite{Bjork2010,Bjork2014,vigna:2014,Vigna2017,
Strub_2019_a,forsyth_2019_c,bjork_book_2021}.

\subsection{Further relationship between EW-ES and  EW-LS problem} \label{sec:equivESLS}

Problem $\EWES(\alpha,\kappa)$ requires specification of the parameter pair $(\alpha,\kappa)$ 
while  problem $\EWLS(\XL,\kappal)$ needs stipulation of parameter pair $(\XL,\kappal)$. 
From Proposition \ref{prop_time} (ii), we learn that, given a value of $\XL$ from equation (\ref{W_star_1}) we can solve
problem Problem $\EWLS(\XL, \kappal)$, which generates a control which is an optimal control for problem
$\EWES(\alpha, \alpha \kappal)$.

However, given an arbitrary value of $\XL$,  for which a 
solution to Problem $\EWLS(\XL, \kappal)$ exists, what is the
relation of the optimal control for this problem to the 
optimal control for Problem $\EWES(\alpha,\kappa)$?

To connect an optimal EW-LS solution  $\mathcal{P}^*_0$  to $\EWES$, we 
define
\begin{equation}\label{alphaK}
\alpha_{\kappal}^*(\XL) = \text{prob}( 
              W_T^* < \XL),\quad  W_{T^*}  
          \text{ is the terminal wealth of } \mathcal{P}^*_0  
       \text{ which solves } \EWLS(\XL, \kappal)
\end{equation}

\begin{remark}[Construction of $\alpha_{\kappal}^*(\XL)$]
Given $(\XL, \kappal)$, and an optimal control $\mathcal{P}_0^* (\XL, \kappal)$ which solves Problem $\EWLS(\XL, \kappal)$, then
we can determine $\alpha_{\kappal}^*(\XL)$ from 
\begin{eqnarray}
    \alpha_{\kappal}^*(\XL)  & = & E_{\mathcal{P}_0^* (\XL, \kappal)}[ {\bf{1}}_{\{W_T*  < \XL\}}]~.
           \label{alpha_kappa_def}
\end{eqnarray}
\end{remark}

{\myred{
To ensure a proper correspondence to $\EWES$, we consider solution to  $\EWLS$ with $0<  \alpha_{\kappal}^*( \XL) <1$, i.e., we consider a restricted domain for $\EWLS$ as:
\begin{eqnarray}
 {\mathcal{D}}_{LS}^{+} & = & \{ (\XL, \kappal)~ | ~ ~ 0<  \alpha_{\kappal}^*( \XL) <1 ~ \text{ and }  \kappal > 0 
               \}
              ~. \label{D_LS_def}
\end{eqnarray}

\begin{assumption}[invertibility of $\alpha_{\kappal}^*(\XL)$] \label{as:inverse}
 The function $\alpha_{\kappal}^*(\XL)$  in \eqref{alphaK} is well defined and is invertible at $(\XL,\kappal )
             \in {{\mathcal{D}}^+_{LS}}$, i.e., 
for any $\XL^{\prime} \neq \XL,  (\XL^{\prime}, \kappal) \in {{\mathcal{D}}^+_{LS}},
             \alpha_{\kappal}^*(\XL^{\prime}) \neq \alpha_{\kappal}^*(\XL)$.
 \end{assumption}
 Note that here we only assume that, for each given $\XL$ and $\kappal$,  
$\EWLS(\XL, \kappal)$ yields a unique probability value $ \alpha_{\kappal}^*(\XL)$  
but  we do not assume uniqueness of {\myblue{the optimal controls}} 
for  $\EWLS(\XL, \kappal)$.

Proposition \ref{prop_ew_ls} below establishes an equivalence of $\EWES$  and $\EWLS$,
under the assumption that the function $\alpha^*_{\kappal}(\XL)$ is  well defined and invertible.

\begin{proposition}[Relationship  between $\EWLS$ and $\EWLS$ for  general  $\XL$]  \label{prop_ew_ls}
Suppose Assumption \ref{as:inverse} holds at ${(\XL, \kappal)} \in {{\mathcal{D}^+_{LS}}}$,
then a solution to $\EWLS(\XL, \kappal)$ is a solution to
 $\EWES(\alpha_{\kappal}^*(\XL), \alpha^*_{\kappal}(\XL)\kappal)$,
with $(\alpha_{\kappal}^*(\XL), \alpha^*_{\kappal}(\XL)\kappal) \in {\mathcal{D}}_{ES}$.
\end{proposition}
\begin{proof}
Consider $\EWLS(\XL, \kappal)$ for a given $(\XL,\kappal) \in {\mathcal{D}^+_{LS}}$. 
Let $\alpha^*_{\kappal}(\XL)$ be defined in \eqref{alphaK}.
Consider $\EWES(\alpha^*,\alpha^* \kappal)$ where $\alpha^* = \alpha^*_{\kappal}(\XL)$. 
Note that by definition of ${\mathcal{D}^+_{LS}}$, we must have
$(\alpha^*,\alpha^* \kappal) \in {\mathcal{D}}_{ES}$.
Proposition \ref{prop_time} (i)  shows that  a solution  of $\EWES(\alpha^*,\kappal \alpha^*)$ 
is a solution to linear shortfall $\EWLS(\XL^*,\kappal)$ problem for  
$\XL^*$  defined in \eqref{W_star_1} with 
prob($W_T^* <\XL^*$) = $\alpha^*_{\kappal}(\XL^*) = \alpha^*$. 
Hence
$$
 \alpha^*_{\kappal}(\XL)= \alpha^*_{\kappal}(\XL^*)=\alpha^*.
$$
Since $\alpha^*_{\kappal}(\XL)$ is invertible, we have  that $\XL=\XL^*$.  
Applying Proposition \ref{prop_time} (ii), using $\XL=\XL^*$ given in \eqref{W_star_1},
a solution to $\EWLS(\XL,\kappal)$  solves 
$\EWES(\alpha^*,\alpha^*\kappal)$, where  $\alpha^*=\alpha^*_{\kappal}(\XL)$.
This completes the proof.
\end{proof}
Applying  Corollary \ref{corollary_1} and Proposition \ref{prop_ew_ls}, we obtain the following Corollary \ref{corollary_2}.

\begin{corollary}\label{corollary_2}
Suppose Assumption \ref{as:inverse} holds for  any  $ (\XL,\kappal) \in {\mathcal{D}^+_{LS}}$.
 Let  ${\mathcal{H}}_{\ES}^*$ and 
${\mathcal{H}}^*_{\LS}$ be defined in \eqref{optSet}.
Then the set  ${\mathcal{H}}_{\ES}^{*} $ of optimal controls for Problem $\EWES$ is  identical to
 the set   ${\mathcal{H}}^{*+}_{\LS} $  of optimal controls for Problem $\EWLS$, where
 $${\mathcal{H}}^{*+}_{\LS} = \{ \mathcal{P}_0^*(\XL,\kappal) \in{\mathcal{H}}_{\LS}^{*}  \text{ and } (\XL,\kappal) \in  {\mathcal{D}^+_{LS}} \}. $$
\end{corollary}
\begin{proof}
From Proposition \ref{prop_time}, any optimal solution $\mathcal{P}_0^*$ of $\EWES$ satisfies $0<\alpha^*_{\kappal}(\XL^*)<1$, $\XL^*$ defined in \eqref{W_star_1}.  Hence we have
 $$\mathcal{H}_{\ES}^* \subset  \mathcal{H}_{\LS}^{*+}. $$
 Conversely, following Proposition \ref{prop_ew_ls},  if $ \mathcal{P}_0^* \in  \mathcal{H}_{\LS}^{*+}$, then 
 $
 \mathcal{P}_0^* \in  \mathcal{H}_{\ES}^{*}
 $, i.e., 
 $
  \mathcal{P}_0^* \in  \mathcal{H}_{\ES}^{*}.
  $ Hence
  $$
  {\mathcal{H}}_{\ES}^{*} \equiv {\mathcal{H}}_{\LS}^{*+}.
  $$
\end{proof}


Corollary \ref{corollary_2} essentially states that,
under Assumption \ref{as:inverse},  the set  of optimal controls associated with all  
$\EWES$ Pareto  efficient frontier curves, 
$(\alpha, \kappa)
 \in {\mathcal{D}}_{ES}$  
is identical to
 the set of optimal controls for all $\EWLS$ Pareto efficient frontier curves with
$(\XL, \kappal) \in {\mathcal{D}}^+_{LS}$. 

\begin{remark}[Significance of Propositions \ref{prop_time} and \ref{W_star_1}]
Proposition \ref{prop_time} (i) shows that any control which solves Problem $\EWES$ solves
the Problem $\EWLS$ with fixed $\XL$ given by
equation (\ref{W_star_1}).  Proposition \ref{prop_time} (ii) informs
us that for certain values of $(\XL, \kappal)$, the optimal control
for Problem $\EWLS$ also solves problem $\EWES$.  This is also true more generally,
for points $(\XL,\kappal)$ satisfying Assumption \ref{as:inverse}.  
It would be interesting to discover conditions on Problem $\EWLS$ which are required to
guarantee that Assumption \ref{as:inverse} holds. We leave this for future work.
\end{remark}
}
}

{\myblue{
\begin{remark}[Numerical experiments: $\EWLS \rightarrow \EWES$]
Problems $\EWLS$ and $\EWES$ are solved numerically, as discussed in
Appendix \ref{Numerical_Appendix}.  
Equation (\ref{alpha_kappa_def}) is approximated
using Monte Carlo methods. 
For values of $(\kappal, \XL)$
such that $\alpha^*_{\kappal}(\XL)$ is small (i.e. $\alpha^*_{\kappal} < .02$),
Proposition \ref{prop_ew_ls}  does not appear to hold.
This may be a result of numerical errors in
approximating the $\alpha$-VAR for small $\alpha$.
\end{remark}
}
}  

\section{Numerical Comparison of Different Risk-Reward Pairs }

\subsection{Data}\label{data_section}
We use data from the Center for Research in Security Prices (CRSP)
on a monthly basis over the 1926:1-2023:12 period.\footnote{More
specifically, results presented here were calculated based on data from
Historical Indexes, \copyright 2024 Center for Research in Security
Prices (CRSP), The University of Chicago Booth School of Business.
Wharton Research Data Services (WRDS) was used in preparing this
article. This service and the data available thereon constitute valuable
intellectual property and trade secrets of WRDS and/or its third-party
suppliers.} Our base case tests use the CRSP US 30 day T-bill for the
bond asset and the CRSP value-weighted total return index for the stock
asset. This latter index includes all distributions for all domestic
stocks trading on major U.S.\ exchanges. All of these various indexes
are in nominal terms, so we adjust them for inflation by using the U.S.\
CPI index, also supplied by CRSP. We use real indexes since investors
funding retirement spending should be focused on real (not nominal)
wealth goals.

We use the parametric model for the real stock index and
real constant maturity bond index described in Appendix \ref{parametric_model_appendix}.

\begin{remark}[Choice of 30-day T-bill for the bond index] 
It might be argued that the bond index should hold longer-dated bonds
such as 10-year Treasuries since this would allow the investor to
harvest the term premium. Long-term bonds  enjoyed high real returns
during 1990-2022.
However, it is unlikely that this will continue to be true
over the next 30 years.  For example, during the
period 1950-1983, long term bonds had negative 
real returns \citep{Hatch_1985}, while short-term T-bills
had positive real returns.
If one imagines that the next 30 years
will be a period with inflationary pressures, this suggests that the
defensive asset should be short-term T-bills. 
Note that the historical real return of short-term T-bills
over 1926:1-2023:12 is approximately zero. Hence our use of T-bills as
the defensive asset is a conservative approach going forward.
\end{remark}

\begin{remark}[Sensitivity to Calibrated Parameters]
Some readers might suggest that the stochastic processes
(\ref{jump_process_stock}-\ref{jump_process_bond}) are
simplistic, and perhaps inappropriate. However, we will
test the optimal strategies (computed assuming processes
(\ref{jump_process_stock}-\ref{jump_process_bond}) ) with calibrated
parameters in Table \ref{fit_params} using bootstrap resampled
historical data (see Section \ref{boot_section} below). The computed
strategy seems surprisingly robust to model misspecification. Similar
results have been noted for the case of multi-period mean-variance
controls \citep{surprise_2021}.
We conjecture that this robustness is due to the self-correcting
nature of feedback controls.
\end{remark}

\subsection{Historical Market}\label{boot_section}

We compute and store the optimal
controls based on the parametric model (\ref{jump_process_stock}-\ref{jump_process_bond}) as for the
synthetic market case. However, we compute statistical quantities using
the stored controls, but using bootstrapped historical return data
directly. In this case, we make no assumptions concerning the stochastic
processes followed by the stock and bond indices. We remind the
reader that all returns are inflation-adjusted. We use the stationary
block bootstrap method \citep{politis1994,politis2004,politis2009,
Cogneau2010,dichtl2016,Scott_2022,Simonian_2022,Cederburg_2022}. 

A key parameter is the expected blocksize. Sampling the data in blocks
accounts for serial correlation in the data series. We use the algorithm
in \citet{politis2009} to determine the optimal blocksize for the bond
and stock returns separately, see Table \ref{auto_blocksize}. We use a
paired sampling approach to simultaneously draw returns from both time
series. In this case, a reasonable estimate for the blocksize for the
paired resampling algorithm would be about $2.0$ years. We will give
results for a range of blocksizes as a check on the robustness of the
bootstrap results. Detailed pseudo-code for block bootstrap resampling
is given in \citet{Forsyth_Vetzal_2019a}.

\begin{table}[tb]
\begin{center}
{\small
\begin{tabular}{lc} \toprule
Data series          & Optimal expected \\
                     & block size $\hat{b}$ (months) \\ \midrule
 Real 30-day T-bill  index    &  50.805 \\
  Real CRSP value-weighted index &  3.17535 \\
\bottomrule
\end{tabular}
}
\end{center}
\caption{Optimal expected blocksize $\hat{b}=1/v$ when the blocksize
follows a geometric distribution $Pr(b = k) = (1-v)^{k-1} v$. The
algorithm in \citet{politis2009} is used to determine $\hat{b}$.
Historical data range 1926:1-2023:12.
\label{auto_blocksize}
}
\end{table}

\subsection{Investment Scenario}
Table \ref{base_case_1} shows our base case investment scenario. We use
thousands of dollars as our units of wealth. For example, a withdrawal
of $40$ per year corresponds to $\$40,000$ per year (all values are
real, i.e.\ inflation-adjusted), with an initial wealth of $1000$ (i.e.\
$\$1,000,000$). This would correspond to the use of the four per cent
rule \citep{Bengen1994}.  Recall that we assume that the investor
has real estate, which is in a separate mental bucket \citep{Shefrin-Thaler:1988}.
Real estate is a hedge of last resort, used to fund required minimum
cash flows \citep{Pfeiffer_2013}.  We assume that the retiree owns
mortgage free real estate worth $\$400,000$, of which $\$200,000$ can
be easily accessed using a reverse mortgage.

\begin{table}[hbt!]
\begin{center}
\begin{tabular}{lc} \toprule
Investment horizon $T$ (years) & 30.0  \\
Equity market index & CRSP Cap-weighted index (real) \\
Bond index & 30-day T-bill (US) (real) \\
Initial portfolio value $W_0$  & 1000 \\
Mortgage free real estate & 400 \\
Cash withdrawal$/$rebalancing times & $t=0,1.0, 2.0,\ldots, 29.0$\\
Maximum withdrawal (per year)   & $ q_{\max} = 60$\\
Minimum withdrawal (per year)   & $ q_{\min} = 30$\\
Equity fraction range & $[0,1]$\\
Borrowing spread $\mu_c^{b}$ & 0.03 \\
Rebalancing interval (years) & 1.0  \\
Target Wealth $\XL$  (EW-xS, $x=\{P,L\}$)  & 0.0 \\
$\alpha$ (EW-ES)               & .05 \\
Stabilization $\epsilon$ (see equation (\ref{PCES_a})) & $ -10^{-4} $ \\
Market parameters & See Table~\ref{fit_params} \\ \bottomrule
\end{tabular}
\caption{Input data for examples.  Monetary units: thousands of dollars.
\label{base_case_1}}
\end{center}
\end{table}

\subsection{Numerical Results}
We use the numerical method described in \citep{forsyth:2022,Forsyth_2024} to compute
the optimal controls, which is based on solving
a Partial Integro-Differential Equation (PIDE), combined
with discretizing the controls and finding the
optimal values by exhaustive search.  A brief overview is given in Appendix \ref{Numerical_Appendix}.

We compute optimal strategies from $\EWES(\alpha, \kappae)$ with $\alpha=0.05$ and $\kappae>0$. For  $\EWLS(\XL,\kappal)$ and $\EWPS(\XL,\kappal)$, we compute optimal strategies with $\XL=0$ and $\kappal>0$. 
\subsection{Convergence}
Appendix \ref{app_convergence} shows convergence as the number of grid
nodes increases, for a single point on the synthetic market EW-LS efficient 
frontier.
It is perhaps more instructive to examine the convergence
of the efficient frontiers.  
Figure \ref{synthetic_frontier_fig} shows the convergence of
the EW-LS frontier, as a function of the PIDE nodes.  The curves for
different numbers of nodes essentially overlap, indicating
satisfactory convergence.  Similar results were obtained for the other strategies.

\begin{figure}[htb!]
\centerline{%
\includegraphics[width=3.0in]{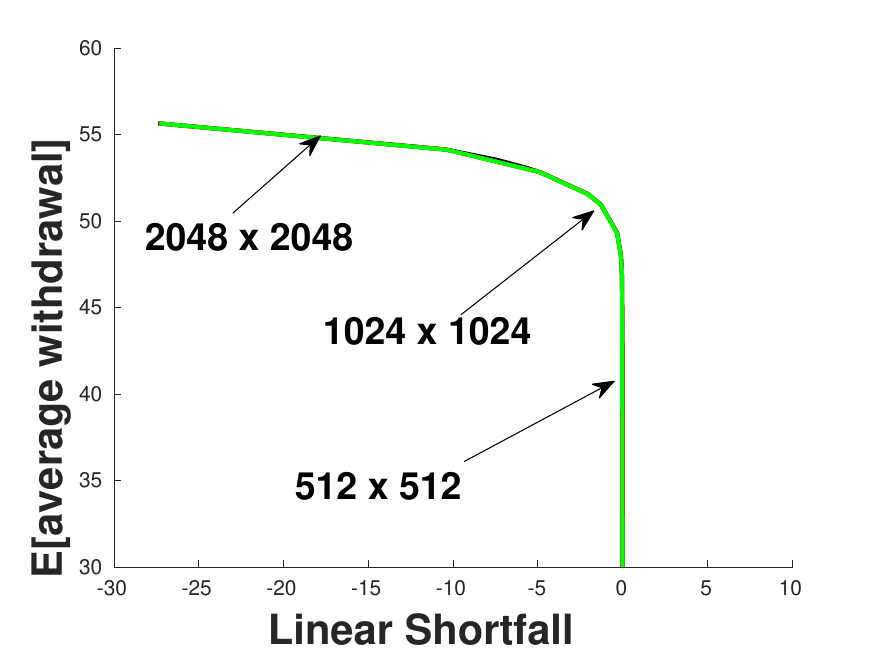}
}
\caption{
EW-LS convergence test.
Real stock index: deflated real capitalization
weighted CRSP, real bond index: deflated 30 day T-bills. Scenario in
Table~\ref{base_case_1}. Parameters in Table~\ref{fit_params}. 
The optimal
control is determined
by solving the  PIDEs as described in
Appendix \ref{Numerical_Appendix}.
Grid refers to the grid used in the algorithm in Appendix
\ref{Numerical_Appendix}: $n_x \times n_b$, where $n_x$ is the number of
nodes in the $\log s$ direction, and $n_b$ is the number of nodes in the
$\log b$ direction. Units: thousands of dollars (real).
The controls are stored, and then the final results
are obtained using a Monte Carlo method,
with $2.56 \times 10^6$ simulations.
Target wealth $\XL=0.0$.
}
\label{synthetic_frontier_fig}
\end{figure}

\subsection{Stabilization term}
In Appendix \ref{Appendix_stablizaton}, we show the effect of
changing the sign of the stabilization term for the EW-PS problem
on the CDF of the final wealth $W_T$.  The stabilization
term has almost no effect on the CDF near $\XL=0$, but does
change the CDF for large values of wealth.  This is 
because the choice of controls for large values of wealth,
as $t \rightarrow T$ is essentially arbitrary.  For large
values of realized wealth, the investor
can choose 100\% bonds or 100\% stocks, and the objective
function will be almost unaffected.\footnote{The 94 year old Warren Buffet, whose net worth
exceeds 145 billion USD, can choose to invest either 100\% in stocks or 100\% in bonds,
and will never run out of savings.}

\subsection{Efficient Frontiers: EW-LS, EW-PS, EW-ES}

In this section,  we compare efficient frontier curves 
computed from $\EWES(\alpha,\kappae)$ with fixed $\alpha = 0.05$ and $\kappae>0$,
and EW-xS$_{t_0}(\XL,\kappal)$ ($x=$P,L) with $\XL=0$ and $\kappal>0$. 
We choose this comparison setting since it seems more 
immediately relevant from a retiree's perspective. 

We assess the performance of these
strategies in the performance domain (EW,LS), (EW,PS), and (EW,ES)  in Figures
\ref{EW_LS_compare_all_fig}, 
\ref{EW_PS_compare_all_fig},  and \ref{EW_ES_compare_all_fig} respectively.

Since all three formulations share the same reward,  
the top left sides  of efficient frontiers from all strategies are expected 
to converge asymptotically (as risk aversion parameter goes to zero)  
in  all performance domains (EW,LS), (EW,PS), and (EW,ES).  

Note that efficient frontier curves  in either performance domain (EW,ES) or 
(EW,LS), from $\EWES(\alpha,\kappa)$ and $\EWLS(\XL,\kappal)$, for fixed $\alpha$ and $\XL$,  
are not expected to coincide, except possibly at single points.

As the risk aversion parameter 
increases, efficient frontiers  on the right side from 
different formulations are expected to deviate from each other more significantly. 
Overall, \ref{EW_LS_compare_all_fig}, Figure \ref{EW_PS_compare_all_fig},  and 
\ref{EW_ES_compare_all_fig}  do demonstrate larger differences in 
efficient frontier curves on the right side, see particularly EW-ES frontiers in Figure \ref{EW_ES_compare_all_fig}.
This confirms that the choice of objective function 
is important in achieving risk control. 
Subsequently we compare and contrast efficient frontiers  in more detail.

\subsubsection{EW-LS}
Figure \ref{EW_LS_compare_all_fig} plots  frontier curves in  the (EW,LS) domain.  We compute  EW-LS,
EW-ES and EW-PS optimal controls, but plot their  (EW,LS) performance measures in the
same figure.
Naturally the frontier curve of the EW-LS control must plot above all the other curves (since the objective
function of EW-LS aligns with the specified measures).
However, it is interesting to see that the EW-ES and EW-PS controls
are not overly suboptimal, relative  to $\EWLS(\XL=05,\kappal)$, using (EW,LS)  criteria.

From Proposition \ref{prop_time}, we expect that there is a point (with target wealth $\XL=0$)  at which the EW-ES and EW-LS curve
coincide.  In Figure \ref{EW_LS_compare_all_fig}, we can see that
this point occurs at $EW \simeq 52.3$.  Figure \ref{EW_LS_compare_all_fig} also shows the results for
the Bengen strategy \citep{Bengen1994}.\footnote{Recall that the recommended policy is to withdraw 4\% of the
initial capital per year, inflation adjusted, and to rebalance to a weight of 50\% stocks annually.
The 4\% withdrawal would correspond to 40 per year for our example.}  We can see that the Bengen
strategy is considerably suboptimal compared to any of the other strategies.  However, it is only fair to 
point out that the Bengen strategy always withdraws $40$ per year (units thousands), while
the other strategies have minimum withdrawals of $30$ per year.

\begin{figure}[htb!]
\centerline{%
\includegraphics[width=3.0in]{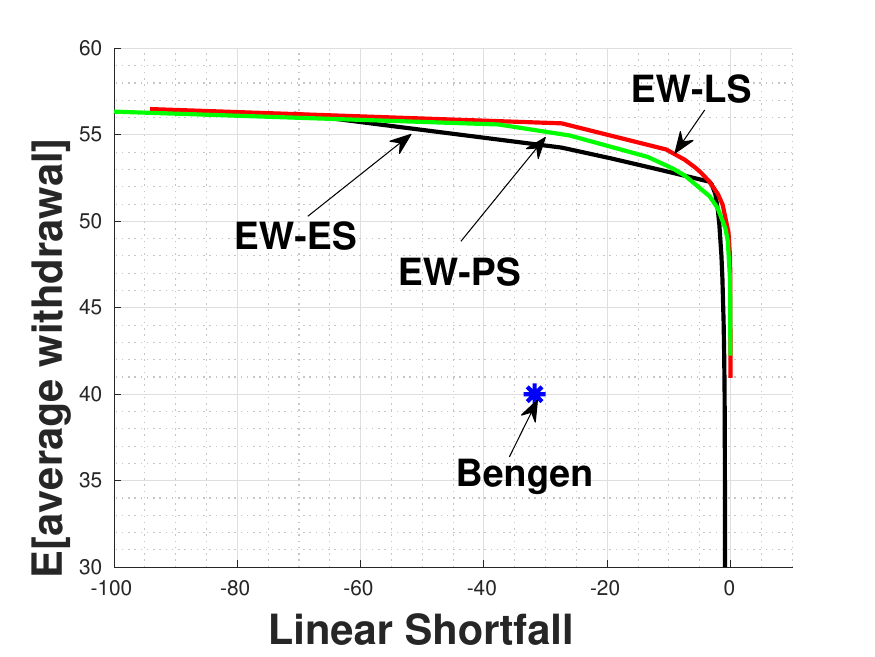}
}
\caption{
EW-LS efficient frontier.
Real stock index: deflated real capitalization
weighted CRSP, real bond index: deflated 30 day T-bills. Scenario in
Table~\ref{base_case_1}. Parameters in Table~\ref{fit_params}. 
Synthetic market.  
Controls computed using EW-PS and EW-ES,   and results
plotted in terms of EW-LS criteria.  The EW-LS frontier plots above
the the controls computed using EW-PS and EW-ES  objective functions.
The Bengen control withdraws 40 per year, and rebalances
annually to 50\% bonds and 50\% stocks. The wealth target level $\XL = 0$ for both $\EWLS$ and $\EWPS$.
}
\label{EW_LS_compare_all_fig}
\end{figure}

\subsubsection{EW-PS}
Figure \ref{EW_PS_compare_all_fig} plots the EW-PS efficient frontier.  Along the
x-axis we plot $Prob[W_T >0 ] = 1 - Prob[W_T<0]$, to produce consistent shapes
for the frontiers.   As before,
we compute the
EW-ES and EW-LS efficient controls, but plot them using PS as a risk measure.
As expected, the EW-PS frontier plots above the other curves (it is, after all,
the efficient strategy according to the $Prob[W_T>0]$ risk measure).

There is somewhat more variation in these curves compared to Figure \ref{EW_LS_compare_all_fig}.
In particular, the EW-PS and EW-LS curves generate $Prob[W_T>0] \simeq 0.9998$  for the largest
values of $\kappa$ (the right hand most point on the curves).  In contrast,
the EW-LS strategy never gets above $Prob[W_T>0] \simeq 0.994$.  We also see that
the EW-LS curve flat-tops  to the left of $EW\simeq 52$.\footnote{If we extend the x-axis to the
left, then, eventually, all three curves meet.  However, these points have an uncomfortably
large  $Prob[W_T] <0$, hence are not of practical interest.}

\subsubsection{EW-ES}

Figure \ref{EW_ES_compare_all_fig} shows the EW-ES efficient frontier.  We also show (EW,ES) measures
for optimal  $\EWPS(0,\kappa)$ and $\EWLS(0,\kappa)$ strategies.   The curves for EW-PS and
EW-LS  are very similar.  However, for $ES > 0$, the EW-ES curve is dramatically
different.  This can be explained as follows.  

The EW-ES strategy moves towards maximizing
ES as $\kappa$ becomes large.  This comes at the expense of decreased $Prob[W_t > 0]$ (from Figure \ref{EW_PS_compare_all_fig}).
On the other hand, the EW-LS and EW-PS strategies have no risk if $W_T > 0$,
so focus entirely on increasing EW, if $W_t > 0$ as  \textcolor{red}{$t \rightarrow T$}.  Effectively,
this means that for the EW-LS and EW-PS strategies, it does not make sense to
consider points in the efficient frontier which are below the {\em knee} of the
curves.  To the left of the {\em knee} of the curves, EW-LS is very close to the
EW-ES curve.

Recall that we have assumed that the retiree can access \$200K using a reverse mortgage
with real estate as collateral.  Consequently, as a rule of thumb, any point on any
frontier which has $ES > -200$ is acceptable from a risk management point of view.
In other words the mean of the worst 5\% of the outcomes can be hedged using
real estate.  In particular, we can see that the Bengen strategy fails this
risk management test.

\begin{figure}[tb]
\centerline{%
\begin{subfigure}[t]{.40\linewidth}
\centering
\includegraphics[width=\linewidth]{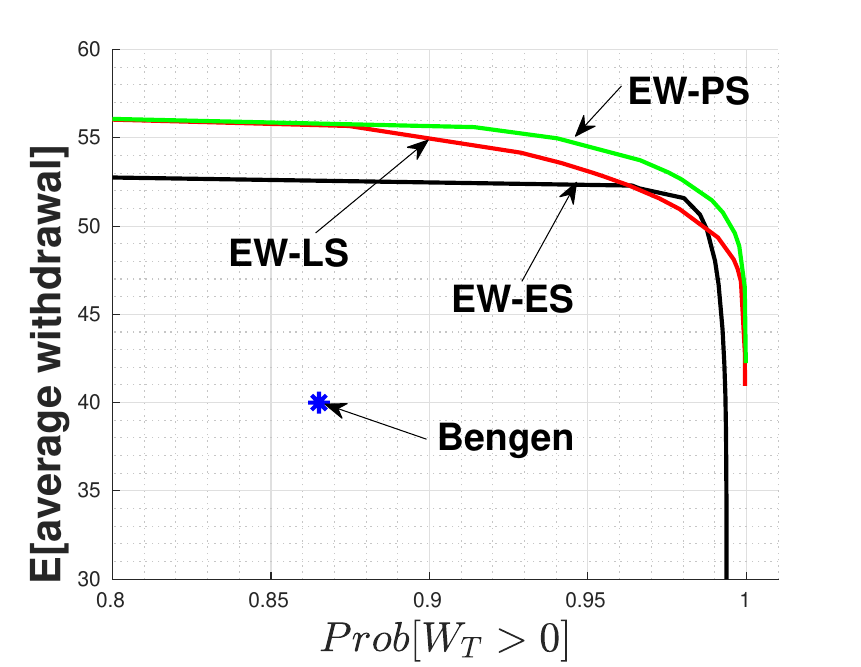}
\caption{EW-PS efficient frontier.  Results for EW-LS and EW-ES also shown.}
\label{EW_PS_compare_all_fig}
\end{subfigure}
\begin{subfigure}[t]{.40\linewidth}
\centering
\includegraphics[width=\linewidth]{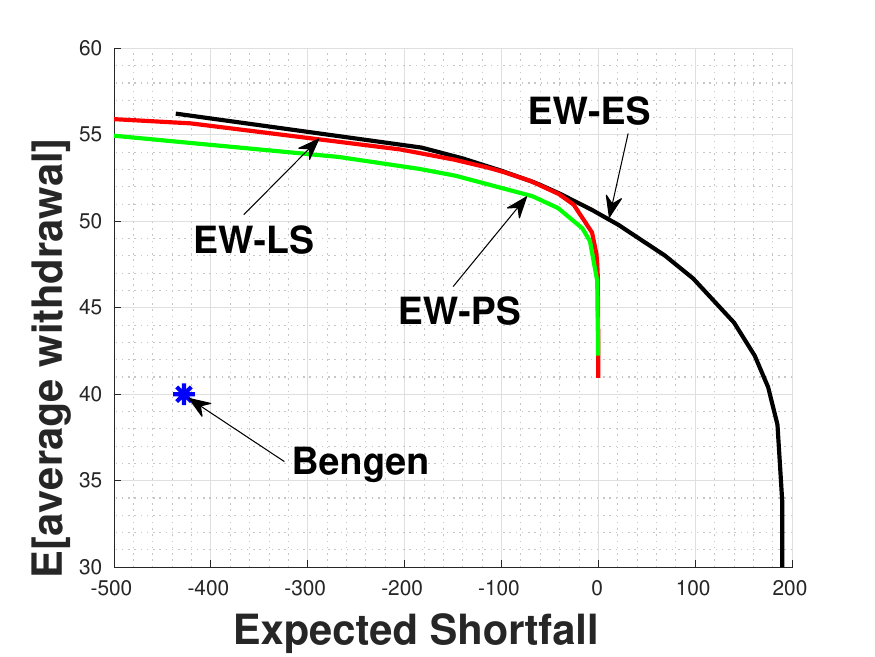}
\caption{
  EW-ES efficient frontier.  Results for EW-LS and EW-PS also shown.
}
\label{EW_ES_compare_all_fig}
\end{subfigure}
}
\caption{
EW-PS and EW-ES efficient frontiers.
Real stock index: deflated real capitalization
weighted CRSP, real bond index: deflated 30 day T-bills. Scenario in
Table~\ref{base_case_1}. Parameters in Table~\ref{fit_params}.
Synthetic market.
The Bengen control withdraws 40 per year, and rebalances
annually to 50\% bonds and 50\% stocks.
Target wealth $\XL=0$ for EW-LS and EW-PS.
}

\label{ES_PS_compare_all_fig}
\end{figure}

\subsubsection{Summary of efficient frontier comparison}
It is relevant  to compare performance of optimal strategies using risk measures which are not directly
included in their respective objective functions.  This helps an investor
understand consequences of implementing a specific strategy 
from in terms of  different but relevant performance measures.

Our first  observation is that in all cases, whatever 
the strategy or risk measure, the Bengen strategy is
significantly sub-optimal.
We also make the following additional observations:

\begin{itemize}

\item
Firstly recall that  we  use the same target
wealth level $\XL=0$ for  both $\EWLS$ and $\EWPS$.  We observe that their frontier curves are close in all 
three measurement domains, (EW,LS), (EW,PS) and (EW,ES), even asymptotically as the 
risk aversion parameter $\kappa \rightarrow +\infty$.  
This suggests that choosing $\EWLS$ also leads to good performance in terms of $\EWPS$ .
Furthermore, since linear shortfall is an expectation of piecewise linear 
shortfall, i.e., $E(\max(W_T-\XL,0))$,  while the probability function is the expectation of 
a discontinuous indicator function, i.e., $E({\bf{1}}_{W_T < \XL})$, consequently
solving $\EWLS$ can be  computationally  preferable to solving
$\EWPS$.

\item Choosing a suitable risk measure as part of the objective function 
which aligns with the desired decumulation goals does matter.  
Different risk measures can  lead to very different performing 
strategies. This is particularly important 
in decumulation. For example, Figure \ref{EW_PS_compare_all_fig} shows 
that the optimal  $\EWES$ strategy is inefficient at minimizing the probability of negative 
terminal wealth. The smallest probability of negative wealth 
achieved is at the expense of steeply diminishing reward.
Similarly Figure \ref{EW_ES_compare_all_fig} shows that, 
with the  wealth target $\XL=0$, the   5\% ES risk associated with the 
optimal  $\EWLS$ and $\EWPS$ strategies as the risk aversion parameter 
$\kappal \rightarrow +\infty$ is far from the optimal 5\%-ES risk achievable.

\item
While Figure \ref{EW_LS_compare_all_fig}  seems to indicate
that all the strategies perform reasonably well in terms of the 
LS risk measure, we note that there is a similar 
steeper drop from the optimal $\EWES$ strategy  
as the risk aversion parameter goes to $+\infty$. 
We further note that the scales of the horizontal axis in  
Figure \ref{EW_LS_compare_all_fig} and 
Figure \ref{EW_PS_compare_all_fig} are very different. 
Optimal $\EWES$ strategies are unable to achieve 
the minimum LS risk and the smallest LS risk strategy 
is achieved at the expense of suboptimal rewards.

\end{itemize}

From Figure \ref{EW_PS_compare_all_fig} we can see that, in terms of the PS risk measure,
EW-LS plots a bit below the optimal EW-PS frontier.  However, the EW-ES curve has
a very unusual behaviour (in terms of PS risk).  Any increase in EW above about 53 causes
a large decrease in $Prob[W_T >0]$.

Turning attention to Figure \ref{EW_ES_compare_all_fig}, in terms of ES risk, all strategies
behave similarly to the left of $ES \simeq 0$.  However, the EW-ES efficient frontier continues
to generate positive ES for $EW < 50$.  Essentially, this is because the EW-ES strategy focuses
on maximizing ES, but at the expense of giving up increases in $Prob[W_T >0]$.
Recall from our previous discussion that the EW-PS and EW-LS strategies
only make sense if we look at points to the left of the {\em knee} of the
curves.

From a practical point of view, it is not clear that maximizing ES when it is positive
is consistent with the retiree's view of risk (i.e. running out of savings).

\subsection{Tests for robustness: bootstrap resampling}
We compute and store the optimal controls for EW-PS, EW-LS and EW-ES objective
functions, based on the parametric market model described in Appendix \ref{parametric_model_appendix}.
We then test these controls using block bootstrap resampling of the market
data in 1926:1-2023:12 \citep{politis1994,politis2004,politis2009,
Cogneau2010,dichtl2016,Scott_2022,Simonian_2022,Cederburg_2022}.

Figure \ref{compare_bootstrap_all} compares the synthetic market results (test and train on the parametric
market model) as well as testing this control on bootstrapped historical data, for all
three objective functions: EW-LS, EW-PS and EW-ES.  The bootstrapped tests are carried out
for a range of expected blocksizes.  In all cases, for all blocksizes, the efficient frontiers are quite
close, indicating that the controls computed using the parametric market model in Appendix \ref{parametric_model_appendix}
are robust to model misspecification.

\begin{figure}[tb]
\centerline{%
\begin{subfigure}[t]{.30\linewidth}
\centering
\includegraphics[width=\linewidth]{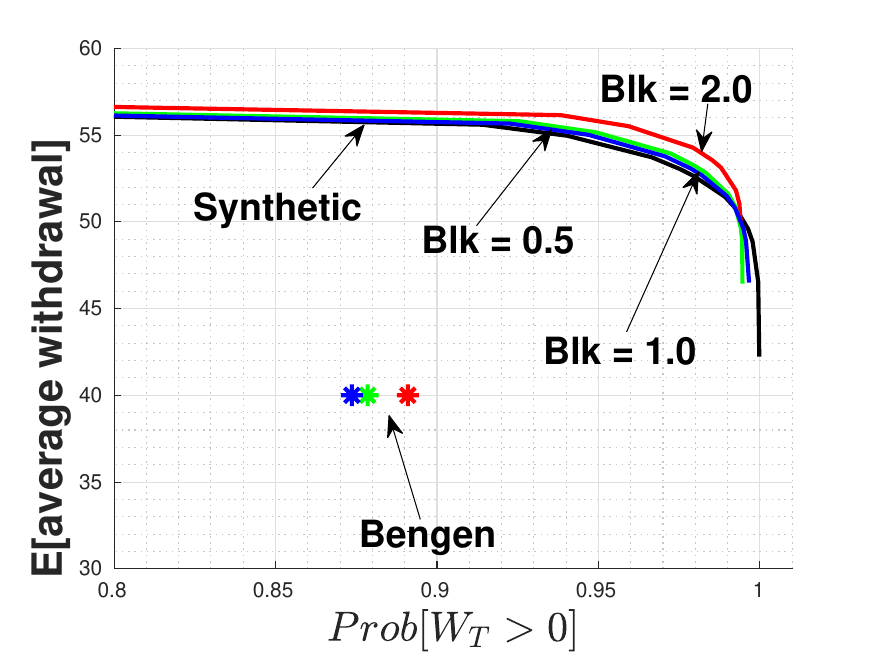}
\caption{EW-PS efficient frontier, bootstrapped simulations. }
\label{EW_PS_boot_fig}
\end{subfigure}
\begin{subfigure}[t]{.30\linewidth}
\centering
\includegraphics[width=\linewidth]{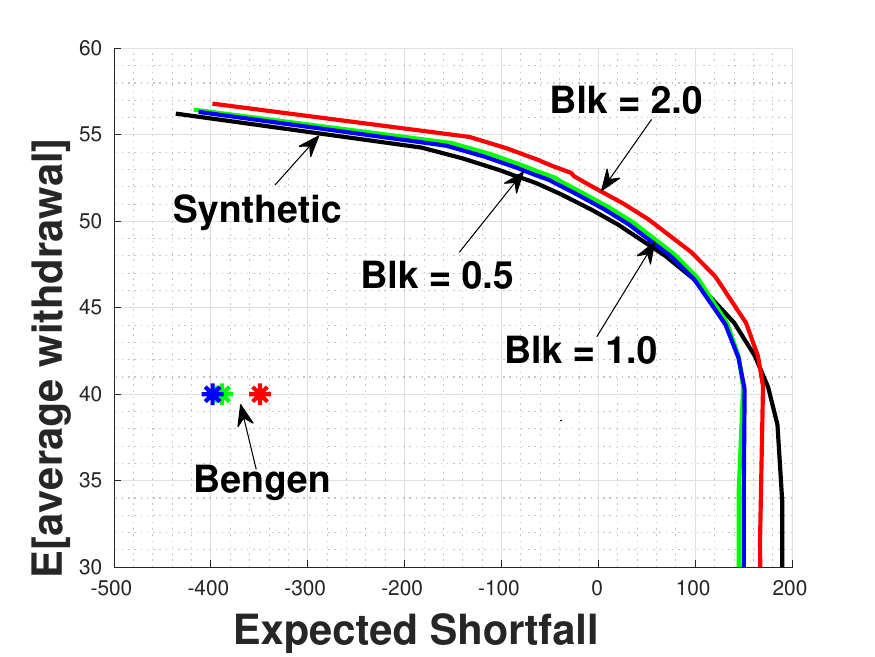}
\caption{
  EW-ES efficient frontier, bootstrapped simulations.
}
\label{EW_ES_boot_fig}
\end{subfigure}
\begin{subfigure}[t]{.30\linewidth}
\centering
\includegraphics[width=\linewidth]{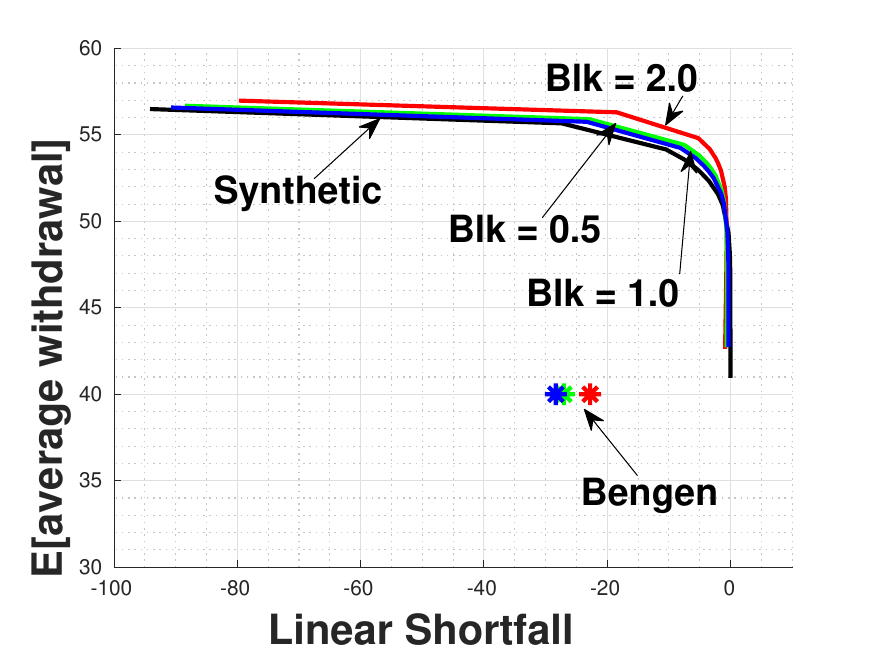}
\caption{
  EW-LS efficient frontier, bootstrapped simulations.
}
\label{EW_LS_boot_fig}
\end{subfigure}
}
\caption{
Optimal controls computed using the synthetic market model.  These
controls tested using bootstrapped historical data.
Expected blocksizes (years) shown.  $10^6$ bootstrap resamples.
Real stock index: deflated real capitalization
weighted CRSP, real bond index: deflated 30 day T-bills. Scenario in
Table~\ref{base_case_1}. Parameters in Table~\ref{fit_params}.
The Bengen control withdraws 40 per year, and rebalances
annually to 50\% bonds and 50\% stocks.  The Bengen results
are also shown for expected blocksizes of $0.5, 1.0, 2.0$ years.
}
\label{compare_bootstrap_all}
\end{figure}

Further insight can be obtained by examining the summary statistics
in Table \ref{CDF_table} (synthetic market) and
Table \ref{CDF_table_boot} (historical market).  It would seem that the EW-LS strategy
is a good compromise, having a relatively small $Prob[W_T <0]$, and
with an expected shortfall close to the optimal value from the
EW-ES solution.  Note that $\XL$ is a byproduct of the optimization
algorithm for the EW-ES problem.  This may not correspond, intuitively,
to the investor's preferences.   For example, as move rightward along the
EW-ES efficient frontier, $\XL$ becomes a large positive value. Any value
of $W_T$ to the left of this point, is regarded (by the objective function)
as a bad outcome, which probably does not correspond to most investor's
concept of risk.

In contrast, $\XL$ is an input parameter
for EW-PS and EW-LS.  In our case, since our main concern is running
out of cash, setting $\XL=0$ is clearly a reasonable choice.

{\myblue{
Note that Table \ref{CDF_table_boot}  shows that
the ES(5\%) result for the EW-LS control (computed in the synthetic market) is actually better than for the
EW-ES control (also computed in the synthetic market) control, when tested
in the historical market.  This suggests that the EW-LS control is more robust
than the EW-ES control.  
}}

\begin{table}[hbt!]
\begin{center}
\begin{tabular}{lcccccc} \toprule
Strategy & $\kappa$ &  $E[\sum_i \qq_i]/M$ & LS($\XL=0$) & ES(5\%) & $Prob[W_T < 0]$ & $\XL$ \\
 \\ \midrule
EW-ES & 0.5925 & 52.97 &    -11.106  &        -102.36 &   .271  &        -31.15 \\
EW-LS & 9.3822 & 52.99 &   -5.3332   &        -106.66 &   .048  &        0.0    \\
EW-PS & 2670.9 & 53.04 &   -9.2747   &        -185.40 &    .027 &       0.0  \\
\bottomrule
\end{tabular}
\caption{
Synthetic market, summary statistics for EW-PS, EW-LS, and EW-ES objective functions, $EW \simeq 53$ for
all strategies.
LS refers to $E[ \min( W_T-\XL, 0) ]$, ES(5\%) is the mean of the worst
five per cent of the outcomes.  $\XL$ is specified for EW-PS and EW-LS,
while it is an outcome of the EW-ES optimization.
Scenario in Table~\ref{base_case_1}. Parameters in Table~\ref{fit_params}.
Units: thousands of dollars (real). $M$ is the total
number of withdrawals (rebalancing dates).
}
\label{CDF_table}
\end{center}
\end{table}

\begin{table}[hbt!]
\begin{center}
\begin{tabular}{lcccc} \toprule
Strategy &  $E[\sum_i \qq_i]/M$ & LS($\XL=0$) & ES(5\%) & $Prob[W_T < 0]$ \\
\midrule
   EW-ES &  53.18 &  -5.4192   & -46.21 & 0.250\\
   EW-LS &  52.93 &  -1.5448 & -30.85 &    0.0226\\
   EW-PS &  53.12 &    -3.705  & -74.02 & 0.0120\\
\bottomrule
\end{tabular}
\caption{
Block bootstrap resampling, summary statistics for EW-PS, EW-LS, and EW-ES objective functions, $EW \simeq 53$ for
all strategies. Blocksize two years, $10^6$ bootstrap resamples.
Optimal controls computed in the synthetic market.
LS refers to $E[ \min( W_T-\XL, 0) ]$, ES(5\%) is the mean of the worst
five per cent of the outcomes.
Scenario in Table~\ref{base_case_1}. Parameters in Table~\ref{fit_params}.
Units: thousands of dollars (real). $M$ is the total
number of withdrawals (rebalancing dates).
}
\label{CDF_table_boot}
\end{center}
\end{table}

Another test of robustness is shown in Table \ref{Rank_boot}.  Here, we rank each strategy, in terms
of performance, according to each risk criteria, in the historical market. All strategies have
approximately the same $EW \simeq 53$.  In this case, we can see that EW-LS is
the clear winner.

\begin{table}[hbt!]
\begin{center}
\begin{tabular}{lcccc} \toprule
    Strategy     & \multicolumn{4}{c}{Rank} \\
           \hline
            &   LS($\XL=0$) & ES(5\%) & $Prob[W_T < 0]$ & Total Score\\
\midrule
   EW-ES &   3   & 2 & 3 & 8\\
   EW-LS &   1  & 1 &  2 &4 \\
   EW-PS &   2  & 3 & 1 & 6\\
\bottomrule
\end{tabular}
\caption{
Ranking of strategies, historical market.  Each strategy is ranked (first, second or third).
Optimal controls computed in the synthetic market.
Total score is the sum of the rows, smaller is better.
$EW \simeq 53$ for all strategies.  Data is from Table \ref{CDF_table_boot}.
Block bootstrap resampling, summary statistics for EW-PS, EW-LS, and EW-ES objective functions.
Blocksize two years, $10^6$ bootstrap resamples.
LS refers to $E[ \min( W_T-\XL, 0) ]$, ES(5\%) is the mean of the worst
five per cent of the outcomes.
Scenario in Table~\ref{base_case_1}. Parameters in Table~\ref{fit_params}.
Units: thousands of dollars (real). 
}
\label{Rank_boot}
\end{center}
\end{table}

\subsection{Comparison with the Bengen strategy}
Consider Figure \ref{EW_PS_boot_fig}, with EW $\simeq 50$.  
This translates to average withdrawals of 5\% of initial wealth
with $Prob[W_T <0] < 1\%$.
Contrast this with the bootstrapped results for the Bengen strategy, where the withdrawals
are 40 per year ($4\%$ of initial wealth (real)), with a probability of failure $>10\%$.

Similarly, Figure \ref{EW_ES_boot_fig} has (EW,ES) = (50,0) for the
EW-ES optimal strategy, compared
with (EW,ES) = $\simeq$ (40, -350) for the Bengen strategy.

Finally, Figure \ref{EW_LS_boot_fig} gives (EW,LS) = (50,0) for the EW-LS optimal
policy, compared with (EW,LS) = $\simeq$ (40, -20) for the Bengen strategy.

Of course, all these comparisons come with the caveat that the Bengen
strategy withdraws a fixed amount per year, while the results for the
optimal strategies are in terms of expected withdrawals.

\section{CDFs of the optimal strategies}

Figure \ref{CDF_boot_synth_all_fig} shows the CDF curves for the final wealth $W_T$ for
all three strategies.  The results are shown for both
the synthetic and historical market.  For each strategy, the point on the efficient frontier
was selected so that $EW \simeq 53$.  It is interesting to observe that all strategies
have similar CDFs for $X > 0$ ($Prob[W_T > 0]$), 
and rapidly increase to the right of this point.  
This indicates that
all strategies are efficient in the sense that there is little unspent wealth
at $t= 30$ years (age 95). This contrasts with the Bengen policy, which has a non-trivial
probability of either running out of cash or  ending up with large unspent wealth.

Figure \ref{CDF_boot_fig_zoom} focuses on the area of the CDF curves near $X = 0$.
Examining the synthetic market results, Figure \ref{CDF_synth_fig_zoom},
we can see that the EW-PS and EW-LS curves behave very similarly near $W_T = 0$, but there is a 
difference in the left tail, as might be expected.  We can see that EW-PS does an
excellent job of producing small $Prob[ W_T < 0]$.  However, this strategy does 
not do well  in the left tail compared with EW-LS.
The EW-ES strategy, on the other hand, has a fairly high probability that $W_T < 0$, compared
with either EW-PS or EW-LS.
However, this is a bit misleading, since the  EW-ES CDF plots below the other strategies for
$X < -40$.   The historical market CDFs, Figure \ref{CDF_boot_fig_zoom}, are qualitatively
similar to the synthetic market curves.

\begin{figure}[tb]
\centerline{%
\begin{subfigure}[t]{.40\linewidth}
\centering
\includegraphics[width=\linewidth]{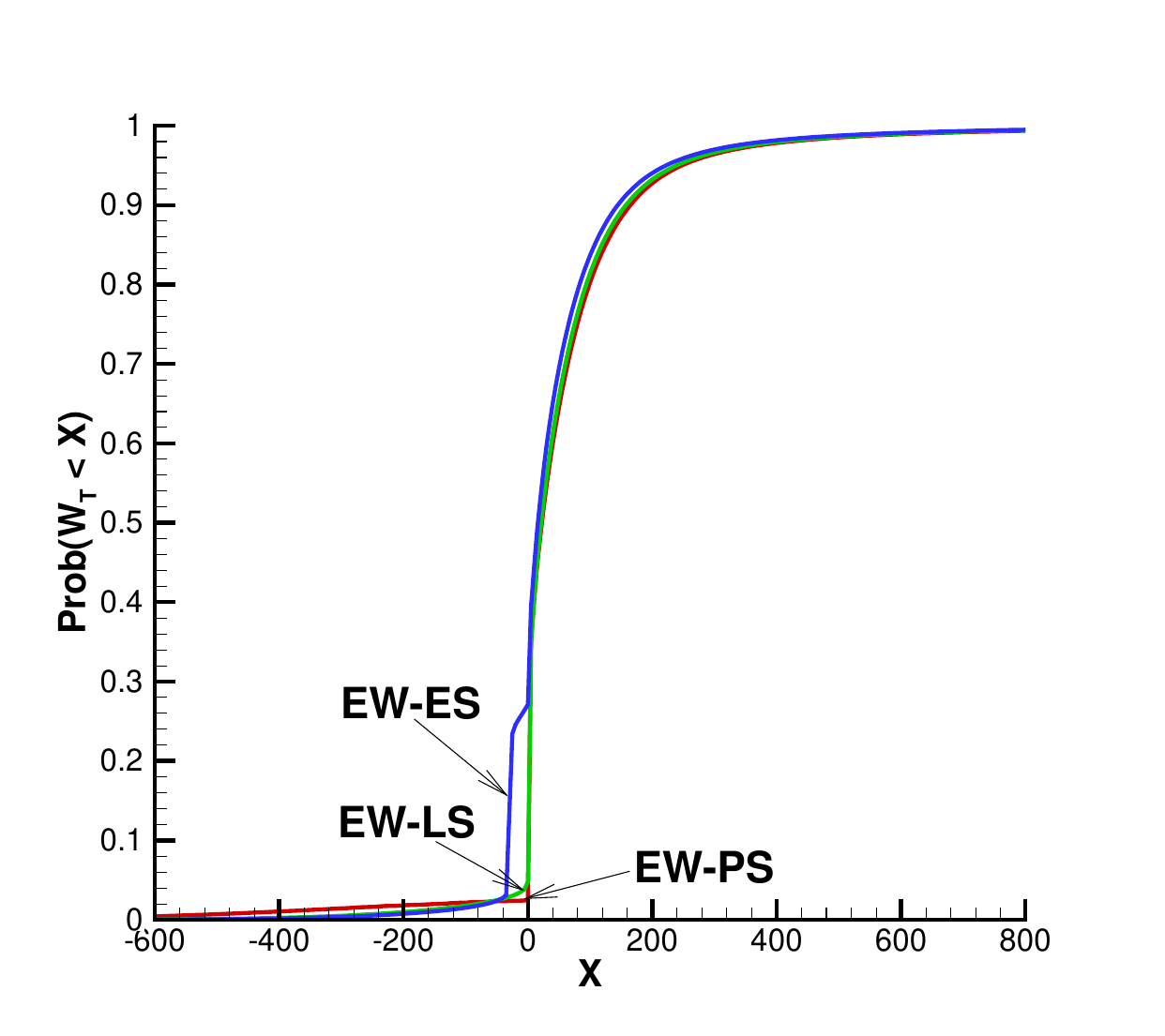}
\caption{CDF curves, synthetic market}
\label{CDF_synth_fig}
\end{subfigure}
\begin{subfigure}[t]{.40\linewidth}
\centering
\includegraphics[width=\linewidth]{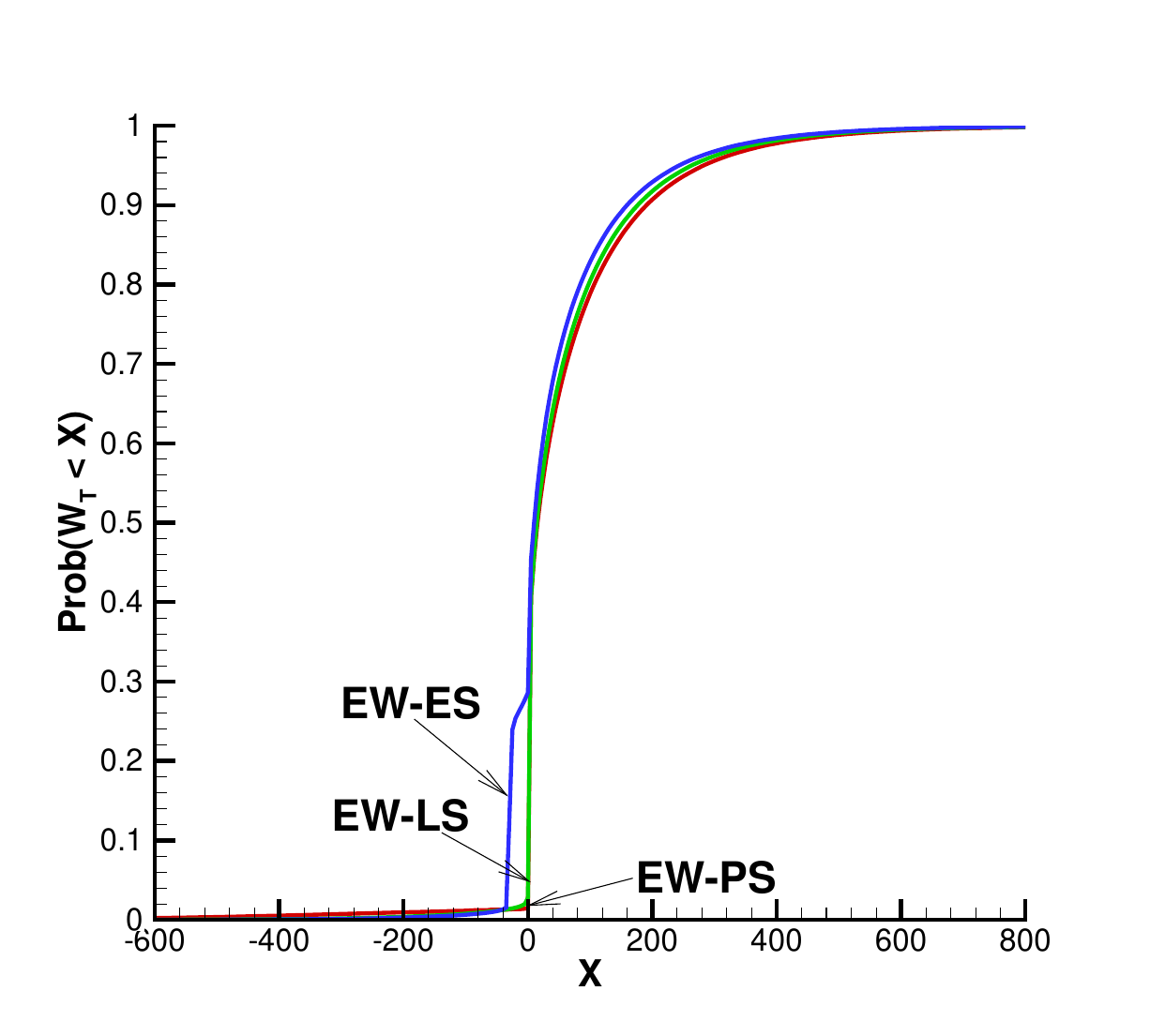}
\caption{
  CDF  curves, bootstrapped historical market.
}
\label{CDF_boot_fig}
\end{subfigure}
}
\caption{
CDF curves, all strategies have the same average EW $\simeq 53$.
Optimal controls computed using the synthetic market model.  Tests in
the synthetic market Figure \ref{CDF_synth_fig} and the historical
market, Figure \ref{CDF_boot_fig} shown.
Expected blocksize: two years.
Real stock index: deflated real capitalization
weighted CRSP, real bond index: deflated 30 day T-bills. Scenario in
Table~\ref{base_case_1}. Parameters in Table~\ref{fit_params}.
}

\label{CDF_boot_synth_all_fig}
\end{figure}

\begin{figure}[tb]
\centerline{%
\begin{subfigure}[t]{.40\linewidth}
\centering
\includegraphics[width=\linewidth]{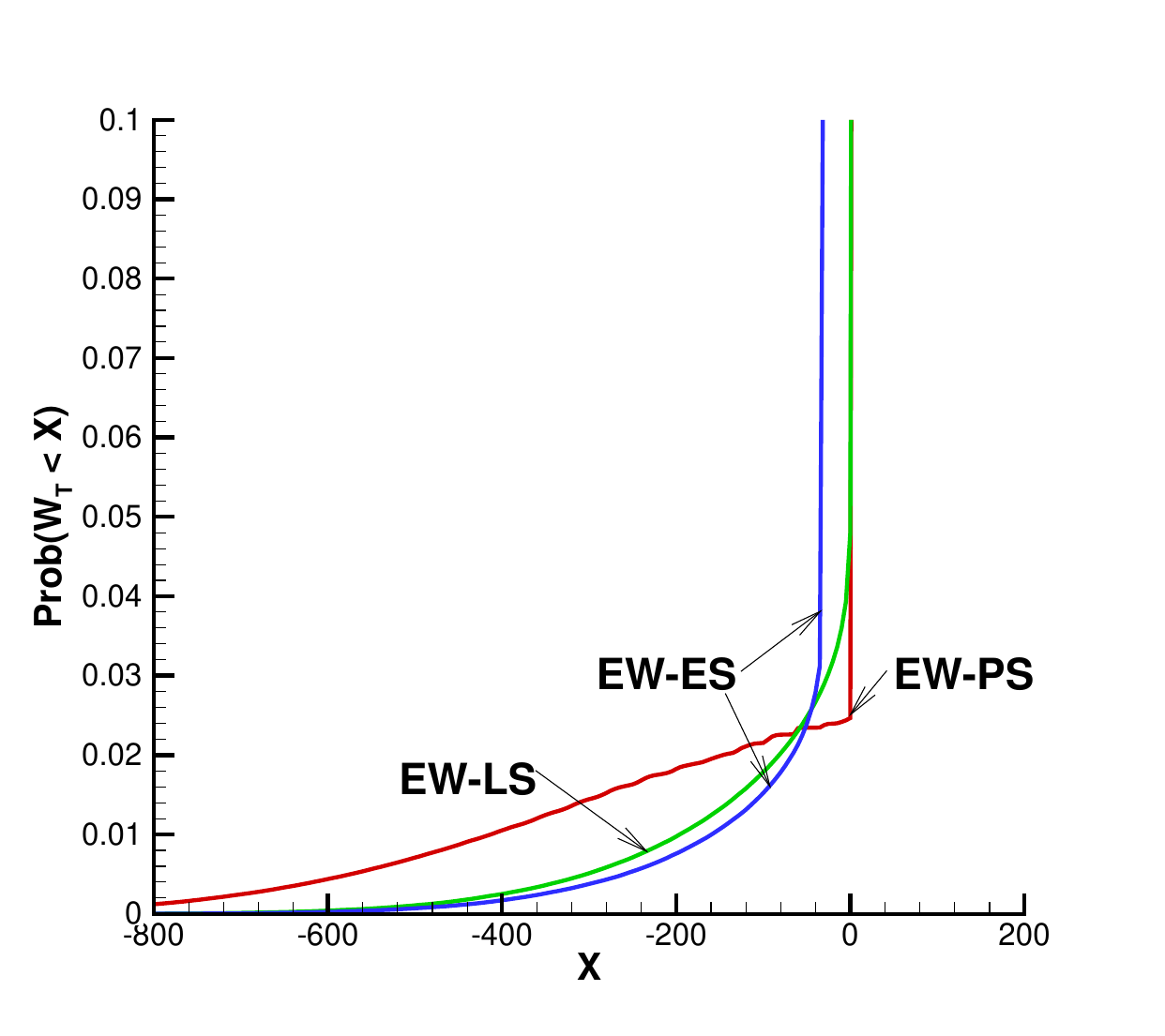}
\caption{CDF curves, synthetic market}
\label{CDF_synth_fig_zoom}
\end{subfigure}
\begin{subfigure}[t]{.40\linewidth}
\centering
\includegraphics[width=\linewidth]{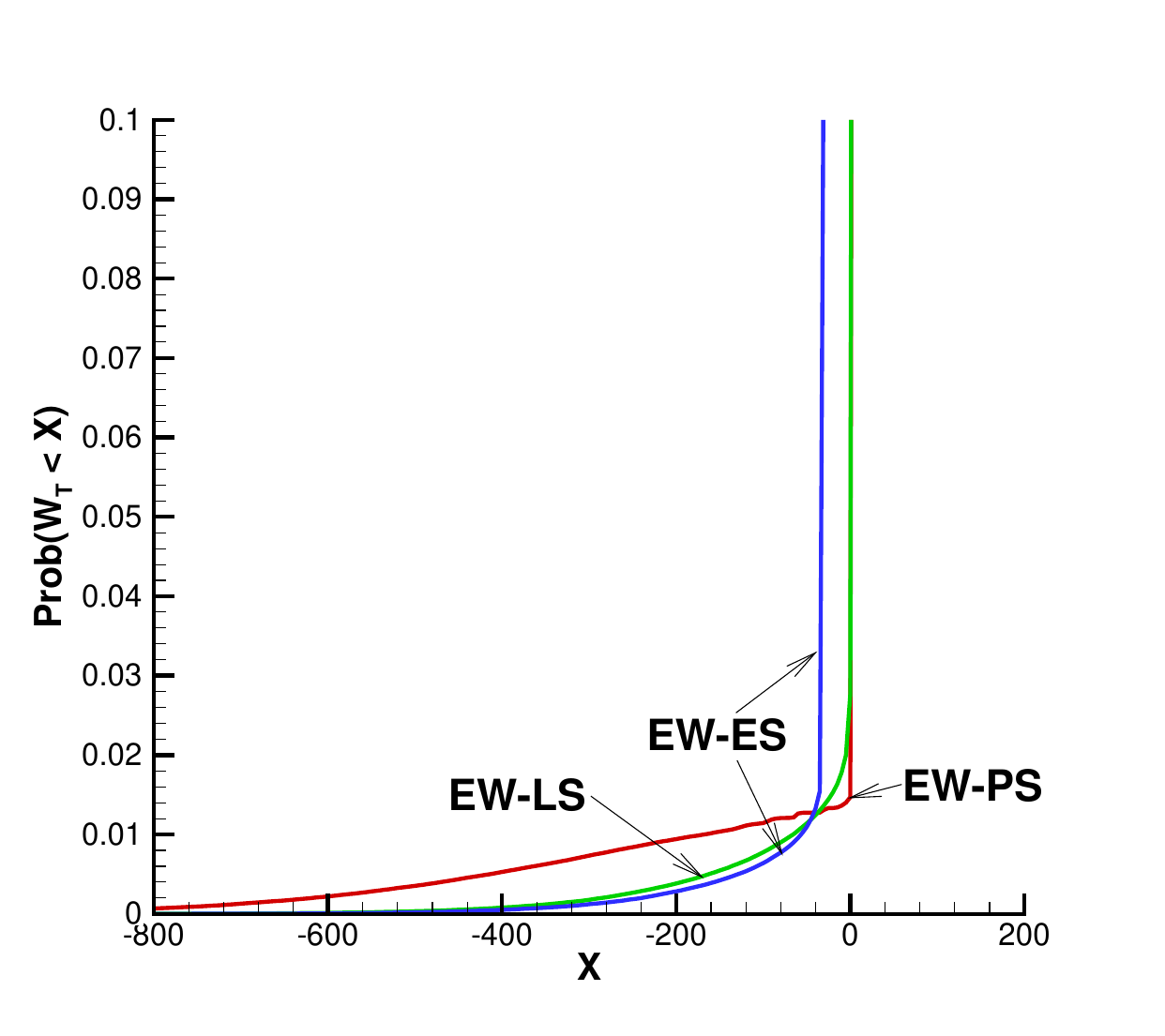}
\caption{
  CDF  curves, bootstrapped historical market.
  $10^6$ bootstrap resamples.
}
\label{CDF_boot_fig_zoom}
\end{subfigure}
}
\caption{
Zoomed plots, CDF curves, all strategies have the same average EW $\simeq 53$.
Optimal controls computed using the synthetic market model.  Tests in
the synthetic market Figure \ref{CDF_synth_fig} and the historical
market, Figure \ref{CDF_boot_fig} shown.
Expected blocksize: two years.
Real stock index: deflated real capitalization
weighted CRSP, real bond index: deflated 30 day T-bills. Scenario in
Table~\ref{base_case_1}. Parameters in Table~\ref{fit_params}.
}

\label{CDF_boot_synth_all_zoom}
\end{figure}

\section{Comments on EW-LS, EW-PS and EW-ES strategies}
The EW-PS optimal control, using PS risk (probability of running out of savings),
seems  at first sight to be an appealing intuitive strategy.  However, 
the CDF of the final wealth shows that this strategy generates
a very fat left tail.  This is simply due to the fact that
PS risk does not weight the amounts less than $\XL$.  

The EW-ES optimal control also has a simple intuitive interpretation.
The ES (mean of the worst 5\% of the outcomes) is a dollar amount
that can be compared with, for example, the retiree's real estate
hedge of last resort.  However, in some cases, the ES can be large
and positive, which does not correspond to what we would normally
think of as risk.  In addition, EW-ES is formally time inconsistent.
There is, of course, an induced time consistent policy, which
is simply the EW-LS control with suitable $\XL$.

The EW-LS control is  trivially time-consistent.  The
investor specified parameter $\XL$ in the EW-LS objective function is
easily interpreted as the disaster level of final wealth.
The EW-LS controls also perform reasonably well using
ES (expected shortfall) or PS (probability of shortfall) as
risk measures.
The EW-LS control is also more robust, when tested in the
historical (bootstrapped) market, compared to the other
strategies.

Consequently, we recommend use of the EW-LS control
for decumulation.

\section{Detailed results: EW-LS, historical market}
Figure \ref{percentile_control_fig} shows the percentiles
of the optimal fraction in stocks, versus time, in the
historical market.  Initially, the fraction in stocks
is a bit less than $0.60$.  The median fraction drops
smoothly down to zero  near year 29.  At the fifth percentile,
complete de-risking occurs at about year 16.  In the case of
poor investment returns,  the allocation to stocks is 0.60-0.80
at the 95th percentile.

Figure \ref{percentile_wealth_fig} shows the wealth percentiles
in the historical market.  We can see that $W_T$ just approaches zero
at the 5th percentile, at year 29.  Again, we remind the reader
that it is assumed that the retiree has real estate which can be
used to fund a shortfall at less than the 5th percentile.  
The expected shortfall at the 5\% level in this case is about $-30$.
Assuming that a reverse mortgage can be obtained for one half
the value of the real estate, this suggests that real estate
valued (in real terms) $ > 60,000$ can manage this risk.

Finally, we can see from Figure \ref{percentile_qplus_fig} that the
median withdrawal rapidly increases to the maximum withdrawal by year one.

\begin{figure}[tb]
\centerline{%
\begin{subfigure}[t]{.33\linewidth}
\centering
\includegraphics[width=\linewidth]{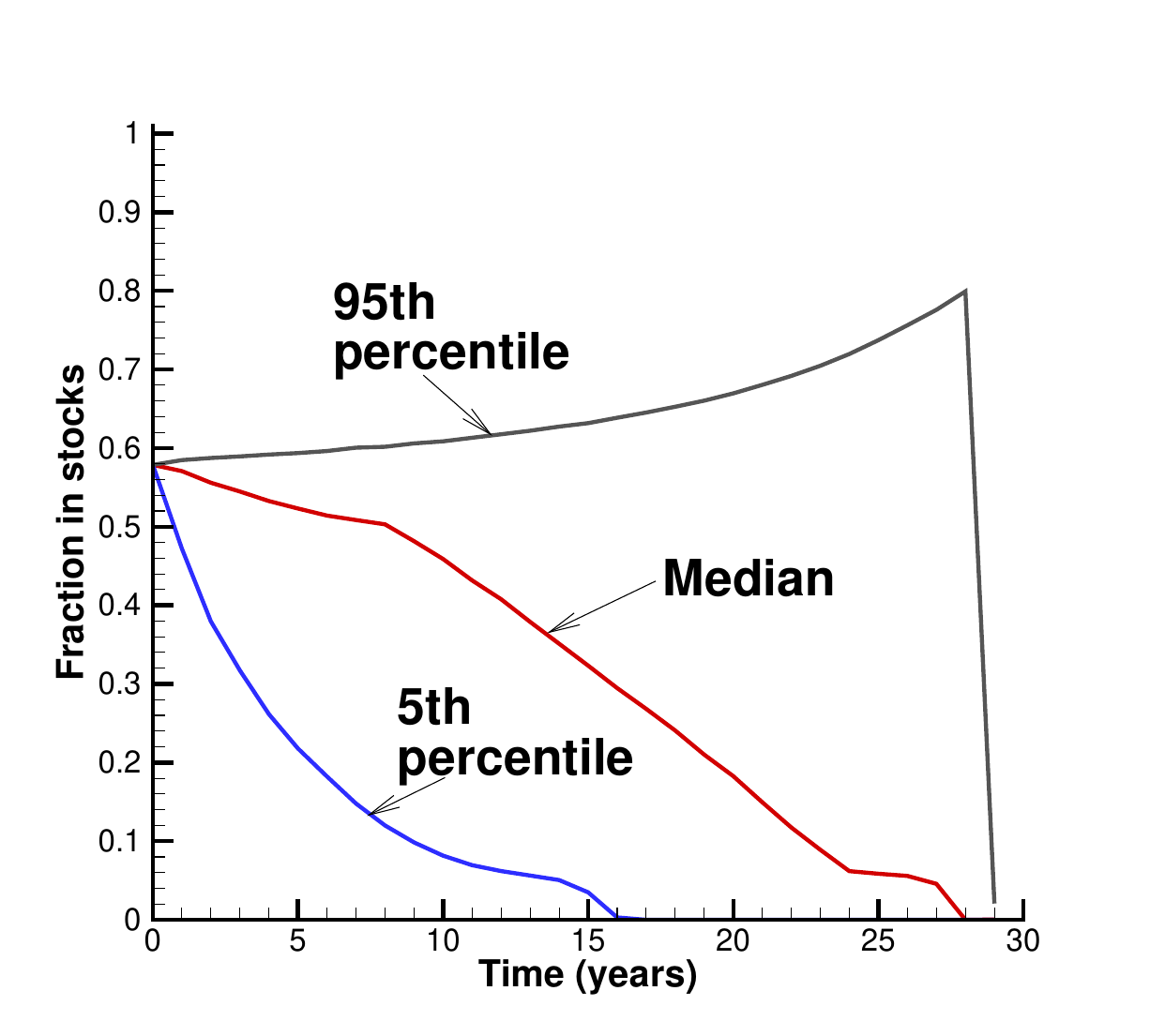}
\caption{Percentiles fraction in stocks}
\label{percentile_control_fig}
\end{subfigure}
\begin{subfigure}[t]{.33\linewidth}
\centering
\includegraphics[width=\linewidth]{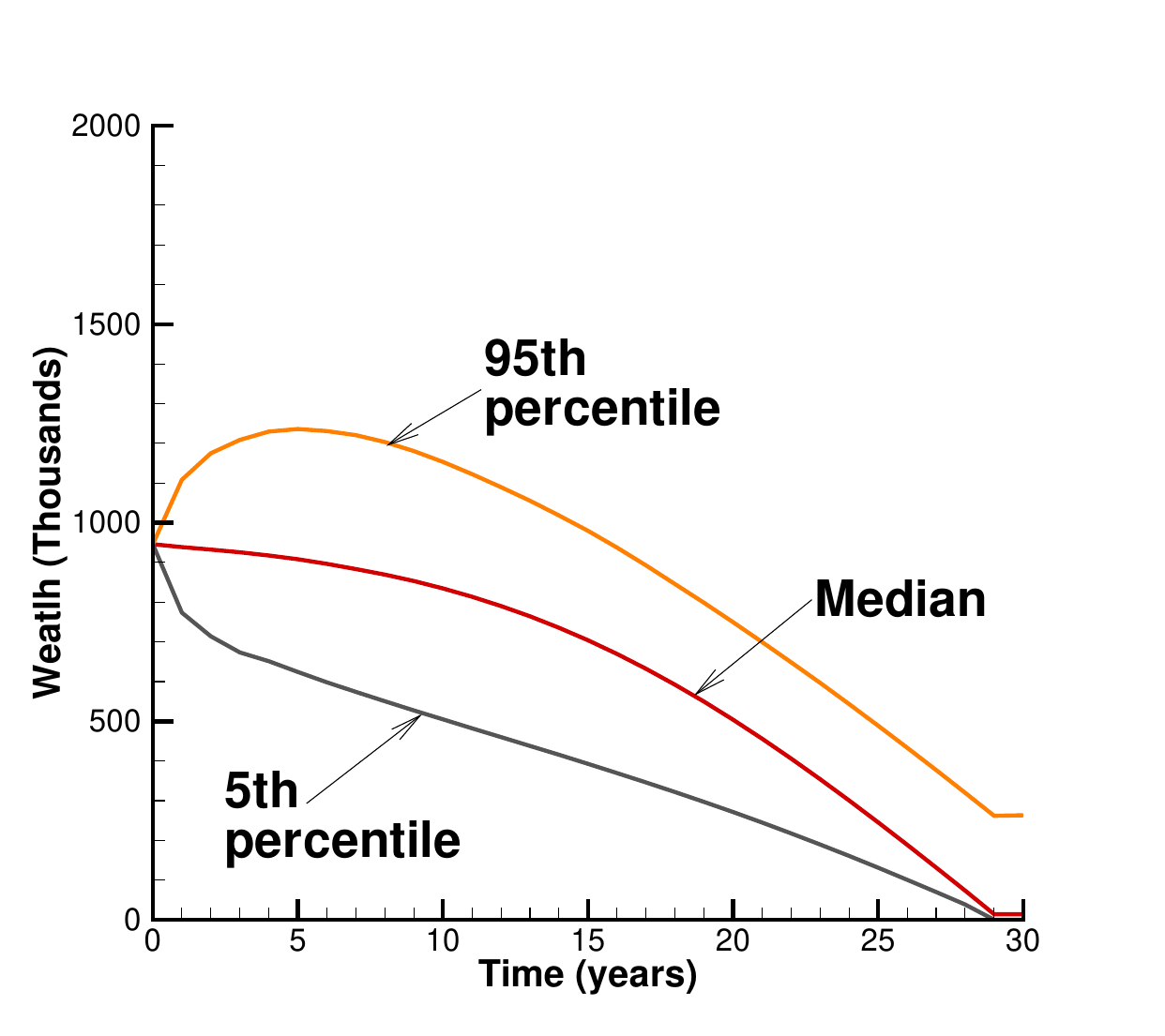}
\caption{Percentiles  wealth}
\label{percentile_wealth_fig}
\end{subfigure}
\begin{subfigure}[t]{.33\linewidth}
\centering
\includegraphics[width=\linewidth]{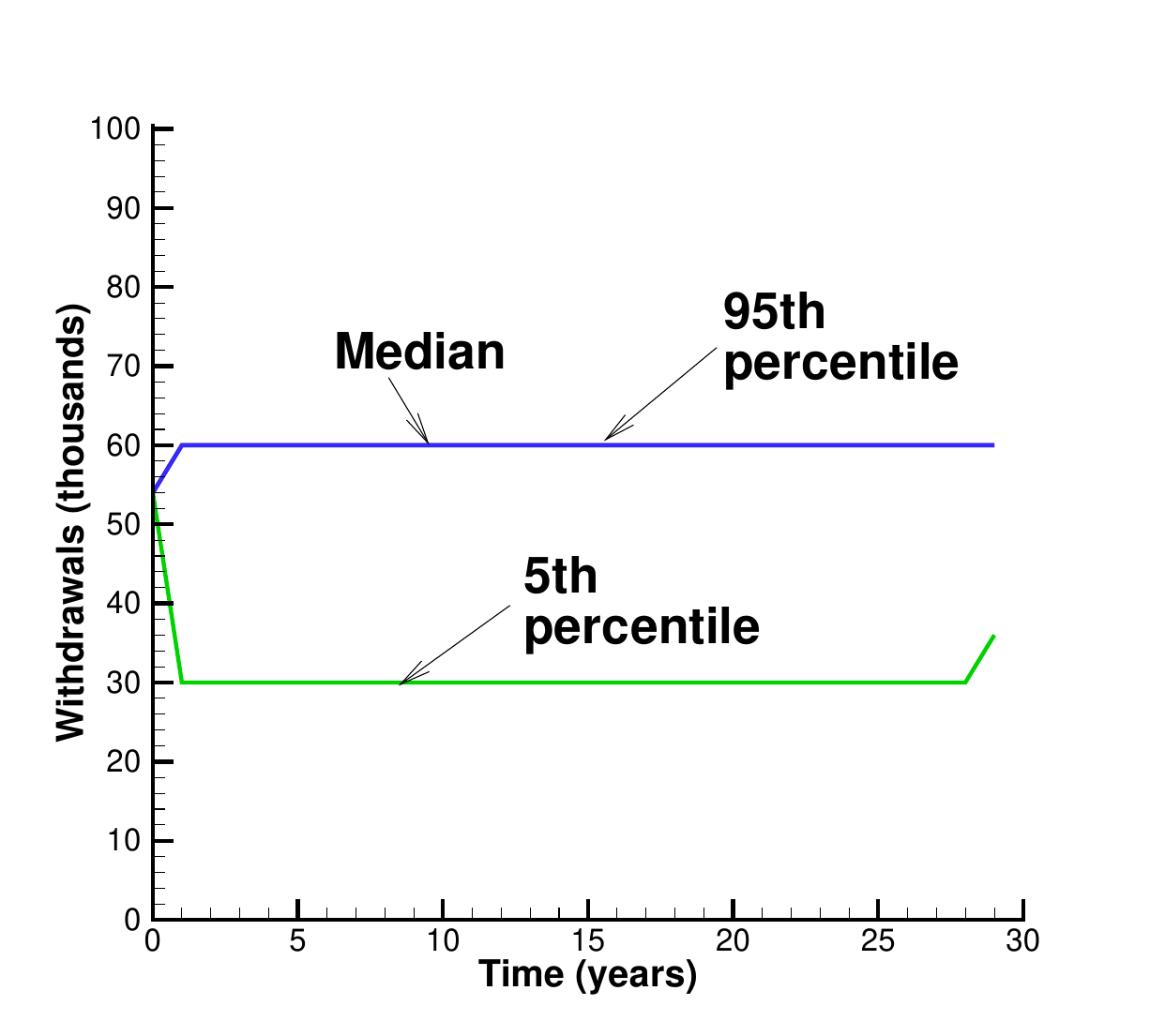}
\caption{Percentiles withdrawals}
\label{percentile_qplus_fig}
\end{subfigure}
}
\caption{
Scenario in Table \ref{base_case_1}. EW-LS control computed from problem
EW-LS Problem~(\ref{PCES_a}). Parameters based on the real CRSP index,
and real 30-day T-bills (see Table~\ref{fit_params}). Control computed
and stored from the Problem~(\ref{PCES_a}) in the synthetic market.
Control used in the historical market, $10^6$ bootstrap samples.
$q_{min} = 30, q_{\max} = 60$ (per year), $EW \simeq 53.0$.
Units: thousands of dollars.
Expected blocksize two years.
}
\label{percentiles_EW_LS_figs}
\end{figure}

The heat maps for the optimal fraction in
stocks and the optimal withdrawals are shown in
Figure \ref{heat_map_EW_LS}.  Figure \ref{heat_map_EW_LS_qplus}
shows that the withdrawal control is approximately bang-bang,
i.e. it is only ever optimal to withdraw $\qq_{\max}$ or
$\qq_{\min}$ and nothing in between.  For an explanation
of this, see \citet{forsyth:2022}.

\begin{figure}[htb!]
\centerline{
\begin{subfigure}[t]{.4\linewidth}
\centering
\includegraphics[width=\linewidth]{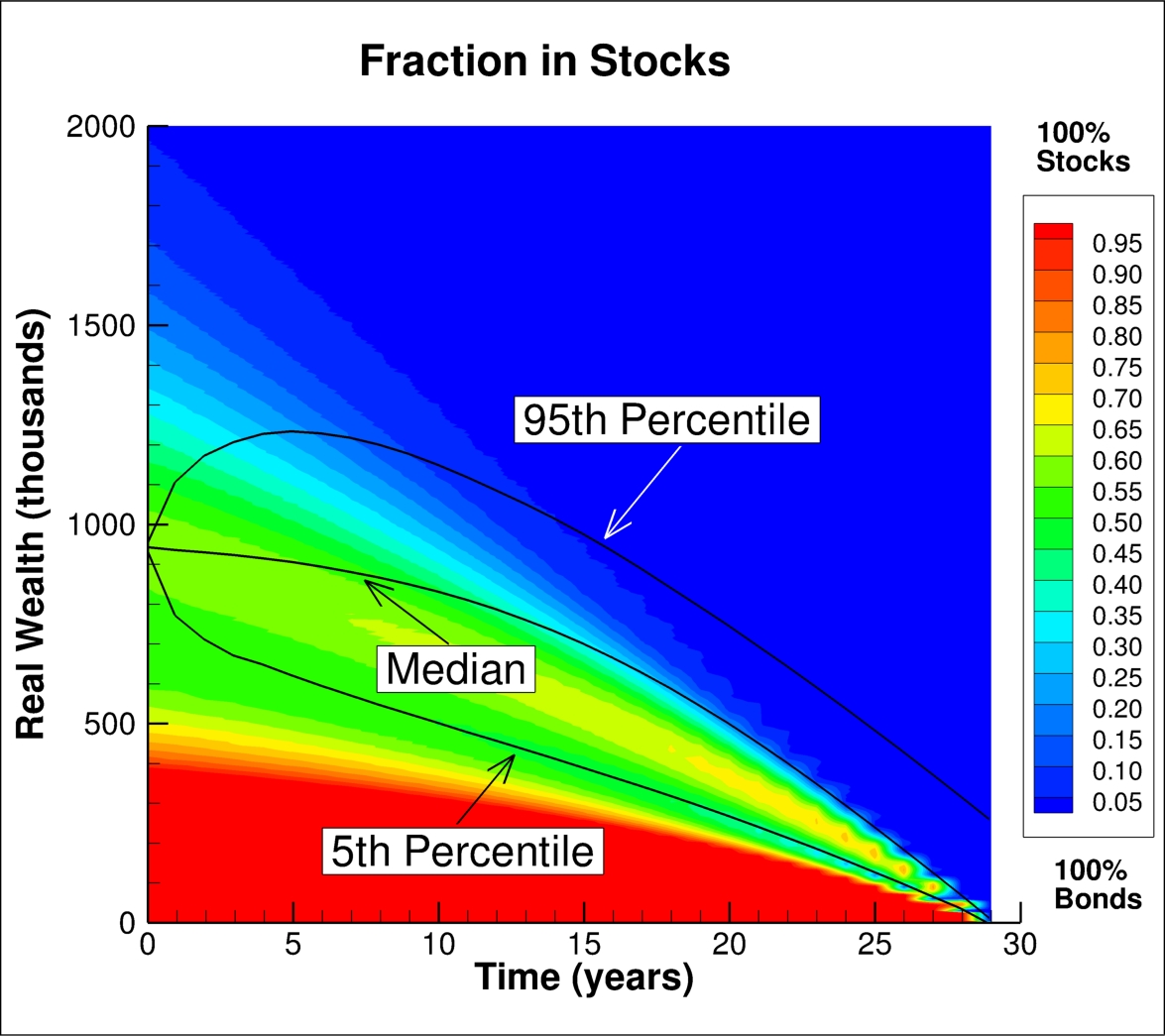}
\caption{Fraction in stocks}
\label{heat_map_EW_LS_allocation}
\end{subfigure}
\hspace{.05\linewidth}
\begin{subfigure}[t]{.4\linewidth}
\centering
\includegraphics[width=\linewidth]{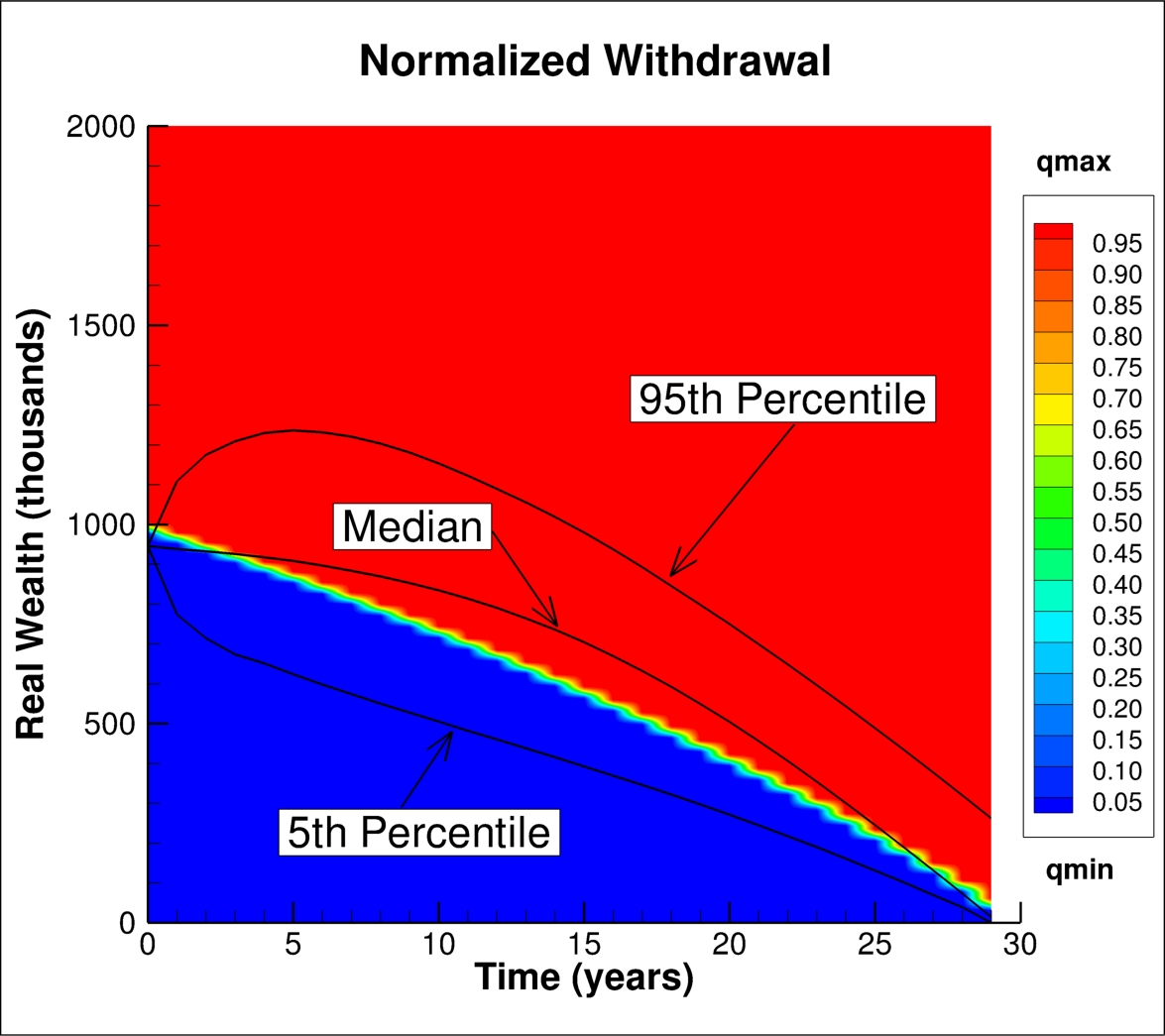}
\caption{Withdrawals}
\label{heat_map_EW_LS_qplus}
\end{subfigure}
}
\caption{
Optimal EW-LS. Heat map of controls: fraction in stocks and withdrawals,
computed from Problem EW-LS~(\ref{PCES_a}). Real
capitalization weighted CRSP index, and real 30-day T-bills. Scenario
given in Table~\ref{base_case_1}. Control computed and stored from the
Problem \ref{PCES_a} in the synthetic market. $q_{min} = 30, q_{\max} =
60$ (per year). $EW \simeq 53.0$.  Percentiles from bootstrapped
historical market.
Normalized withdrawal $(q - q_{\min})/(q_{\max} - q_{\min})$. Units:
thousands of dollars.
}
\label{heat_map_EW_LS}
\end{figure}

\section{Conclusions}
As noted in \citet{Anarkulova_2022_a}, retirees and wealth advisors
demonstrate a revealed preference for spending rules for
decumulation of DC pension plans.
Almost all previous work on spending rules  postulates
heuristic strategies and  tests these rules using
historical data.

We follow a different methodology here.  We determine
the spending rules as the solution of an optimal
stochastic control problem.  The control problem is
solved numerically, based on a parametric model of long term stock
and bond returns.  

For an optimal control problem, the first
order of business is to specify the objective function,
in terms of risk and reward.  Since we allow variable
withdrawals (subject to maximum and minimum constraints)
we define reward as the total expected (real) withdrawals
over  a 30 year retirement (EW).  

We assess and compare  ES, LS, and PS risk measures.
We establish mathematically that,
under certain assumptions, 
the set  of optimal controls associated with all expected reward and 
expected shortfall (EW-ES) Pareto efficient frontier curves  {\myblue{ is identical to}} 
the set  of optimal controls for all  expected reward 
and linear shortfall (EW-LS) Pareto efficient frontier curves.
This has the consequence
that the set of optimal controls for $\EWES$ are time
consistent under the $\EWLS$ risk measure.

Based on our analysis and
computational assessment of various risk measures, we conclude that
risk as measured by linear shortfall LS, i.e.
linearly weighting the final wealth below zero,
is an appropriate risk measure. 

As noted, the optimal EW-LS control is computed
using a parametric market model. However, this
control has been tested out-of-sample using
block bootstrap resampling of historical data. 
These tests show that the optimal control is robust
to parameter misspecification.

Bootstrap resampling of historical data shows that
the 4\% rule (initial capital: one million, withdrawing 4\% real of initial capital
per year) has a probability of failure $> 10 \%$,
and expected shortfall ES(5\%) $<-\$350,000$.
In contrast, under bootstrap resampling tests, the EW-LS optimal control can
withdraw 5\% of initial wealth annually, on average, (adjusted
for inflation) with a 98\% probability of success,
with an ES(5\%) $\simeq -\$15,000$.

The EW-LS controls are dynamic.  Both withdrawal amounts
and stock allocation depend on the realized portfolio
wealth (and time to go).  However, the controls are summarized
as easy to interpret heat maps, which makes implementation
of these optimal controls straightforward.

Finally, we note that the optimal controls can be computed
directly from the bootstrapped resampled data, without
specifying a parametric model of the underlying stock
and bond processes.  This requires use of
machine learning techniques \citep{Ni_2022,beating_benchmark_2023,van_staden_2025a}.
These methods also allow use of more assets in terms of
investment choices.  We leave further study of machine
learning techniques in the context of DC decumulation
for future work.

\section{Acknowledgements}
Forsyth's work was supported by the Natural Sciences and Engineering
Research Council of Canada (NSERC) grant RGPIN-2017-03760. Li's work
was supported by a  Natural Sciences and Engineering
Research Council of Canada (NSERC) grant RGPIN-2020-04331.

\section{Declaration}
The authors have no conflicts of interest to report.

\appendix
\section*{Appendix}

\section{Parametric Model}\label{parametric_model_appendix}
We assume that the investor has access to two funds: a broad market
stock index fund and a constant maturity bond index fund. The investment
horizon is $T$. Let $S_t$ and $B_t$ respectively denote the real
(inflation adjusted) \emph{amounts} invested in the stock index and
the bond index respectively. In general, these amounts will depend
on the investor's strategy over time, as well as changes in the real
unit prices of the assets. In the absence of an investor determined
control (i.e.\ cash withdrawals or rebalancing), all changes in $S_t$ and
$B_t$ result from changes in asset prices. We model the stock index as
following a jump diffusion.

In addition, we follow the usual practitioner approach and directly model
the returns of the constant maturity bond index as a stochastic process
\citep[see, e.g.\@][]{Lin_2015,mitchell_2014}.   
As in \citet{mitchell_2014}, we assume that the constant maturity bond
index follows a jump diffusion process. Empirical justification for this
can be found in \citet{Forsyth_Arva_2022}, Appendix A.

Let $ S_{t^-} = S(t - \epsilon), \epsilon \rightarrow 0^+$, i.e.\
$t^-$ is the instant of time before $t$, and let $\xi^s$ be a random
number representing a jump multiplier. When a jump occurs, $S_t = \xi^s
S_{t^-}$. Allowing for jumps permits modelling of non-normal asset
returns. We assume that $\log(\xi^s)$ follows a double exponential
distribution \citep{kou:2002,Kou2004}. If a jump occurs, $u^s$ is the
probability of an upward jump, while $1-u^s$ is the chance of a downward
jump. The density function for $y = \log (\xi^s)$ is
{\color{black}
\begin{linenomath*}
\begin{equation}
f^s(y) = u^s \eta_1^s e^{-\eta_1^s y} {\bf{1}}_{y \geq 0} +
       (1-u^s) \eta_2^s e^{\eta_2^s y} {\bf{1}}_{y < 0}~.
\label{eq:dist_stock}
\end{equation}
\end{linenomath*}
}
We also define
\begin{linenomath*}
\begin{equation}
\gamma^s_{\xi} = E[ \xi^s -1 ] =
  \frac{u^s \eta_1^s}{\eta_1^s - 1} + 
  \frac{(1 - u^s)\eta_2^s}{\eta_2^s + 1} - 1 ~.
\end{equation}
\end{linenomath*}
In the absence of control, $S_t$ evolves according to
\begin{equation}
\frac{dS_t}{S_{t^-}} = 
  \left(\mu^s -\lambda_\xi^s \gamma_{\xi}^s \right) \, dt + 
  \sigma^s \, d Z^s +  d\left( \ds \sum_{i=1}^{\pi_t^s} (\xi_i^s -1) \right) ,
\label{jump_process_stock}
\end{equation}
where $\mu^s$ is the (uncompensated) drift rate, $\sigma^s$ is the
volatility, $d Z^s$ is the increment of a Wiener process, $\pi_t^s$ is
a Poisson process with positive intensity parameter $\lambda_\xi^s$,
and $\xi_i^s$ are i.i.d.\ positive random variables having distribution
(\ref{eq:dist_stock}). Moreover, $\xi_i^s$, $\pi_t^s$, and $Z^s$ are
assumed to all be mutually independent.

Similarly, let the amount in the bond index be $B_{t^-} = B(t -
\epsilon), \epsilon \rightarrow 0^+$. In the absence of control, $B_t$
evolves as
\begin{equation}
\frac{dB_t}{B_{t^-}} = \left(\mu^b -\lambda_\xi^b \gamma_{\xi}^b  
   + \mu_c^b {\bf{1}}_{\{B_{t^-} < 0\}}  \right) \, dt + 
  \sigma^b \, d Z^b +  d\left( \ds \sum_{i=1}^{\pi_t^b} (\xi_i^b -1) \right) ,
\label{jump_process_bond}
\end{equation}
where the terms in equation (\ref{jump_process_bond}) are defined
analogously to equation (\ref{jump_process_stock}). In particular,
$\pi_t^b$ is a Poisson process with positive intensity parameter
$\lambda_\xi^b$, and $\xi_i^b$ has distribution
\begin{linenomath*}
\begin{equation}
f^b( y= \log \xi^b) = u^b \eta_1^b e^{-\eta_1^b y} {\bf{1}}_{y \geq 0} +
       (1-u^b) \eta_2^b e^{\eta_2^b y} {\bf{1}}_{y < 0}~,
\label{eq:dist_bond}
\end{equation}
\end{linenomath*}
and $\gamma_{\xi}^b = E[ \xi^b -1 ]$. $\xi_i^b$, $\pi_t^b$, and
$Z^b$ are assumed to all be mutually independent. The term $\mu_c^b
{\bf{1}}_{\{B_{t^-} < 0\}}$ in equation~(\ref{jump_process_bond})
represents the extra cost of borrowing (the spread).

The diffusion processes are correlated, i.e.\ $d Z^s \cdot d Z^b =
\rho_{sb}~ dt$. The stock and bond jump processes are assumed mutually
independent. See \citet{forsyth_2020_a} for justification of the
assumption of stock-bond jump independence.

We use the threshold technique
\citep{mancini2009,contmancini2011,Dang2015a} to estimate the parameters
for the parametric stochastic process models. Since the index data is in
real terms, all parameters reflect real returns. Table \ref{fit_params}
shows the results of calibrating the models to the historical data.
The correlation $\rho_{sb}$ is computed by removing any returns which
occur at times corresponding to jumps in either series, and then
using the sample covariance. Further discussion of the validity of
assuming that the stock and bond jumps are independent is given in
\citet{forsyth_2020_a}.

{\small
\begin{table}[hbt!]
\begin{center}
\begin{tabular}{cccccccc} \toprule[1pt]
 CRSP & $\mu^s$ & $\sigma^s$ & $\lambda^s$ & $u^s$ &
  $\eta_1^s$ & $\eta_2^s$ & $\rho_{sb}$ \\ \midrule
       & 0.087323  &  0.147716&   0.316326  &  0.225806 & 4.3591 & 5.53370 & 0.095933\\
 \midrule[1pt]
30-day T-bill & $\mu^b$ & $\sigma^b$ & $\lambda^b$ & $u^b$ &
  $\eta_1^b$ & $\eta_2^b$ & $\rho_{sb}$ \\ \midrule
        & 0.0032 & 0.0140& 0.3878  &   0.3947 &  61.5350 & 53.4043 &0.095933
\\
\bottomrule[1pt]
\end{tabular}
\end{center}
\caption{Parameters for parametric market models (\ref{jump_process_stock}
and (\ref{jump_process_bond}, fit to CRSP data (inflation
adjusted) for 1926:1-2023:12.  \label{fit_params} }
\end{table}
}

\section{Numerical Techniques}\label{Numerical_Appendix}
We solve problems (\ref{PCES_a}) using the techniques described in
detail in \citet{forsythlabahn2017,forsyth_2019_c,forsyth:2022}. We
give only a brief overview here.

We localize the infinite domain to $(s,b) \in [s_{\min}, s_{\max}]
\times [b_{\min}, b_{\max}]$, and discretize $[b_{\min},b_{\max}]$
using an equally spaced $\log b$ grid, with $n_b$ nodes. Similarly, we
discretize $[s_{\min}, s_{\max}]$ on an equally spaced $\log s$ grid,
with $n_s$ nodes.
For case $b<0$, we define a reflected grid $b^{\prime} = -b$, with
the $n_b \times n_s$ nodes.  This represents
the insolvent case nodes. The PIDE for $b^{\prime} >0$
has the same form as for $b > 0$.
This idea can be used more generally if leverage is
permitted, which we do not explore in this work.
Localization errors are minimized using the domain
extension method in \citet{forsythlabahn2017}.

At rebalancing dates, we solve the local optimization problem
by discretizing $(\qq (\cdot), \pp(\cdot)
)$ and using exhaustive search. Between rebalancing dates,
we solve a two dimensional partial integro-differential
equation (PIDE) using Fourier methods
\citep{forsythlabahn2017,forsyth:2022}. Finally, in the
case of EW-ES, the outer optimization over $\XL$ is
solved using a one-dimensional method.

We used the value $\epsilon = -10^{-4}$ in equation (\ref{PCES_a}),
which forces the investment strategy to be bond heavy if the remaining
wealth in the investor's account is large, and $t \rightarrow T$. Using
this small value of gave the same results as $\epsilon = 0$ for the
summary statistics, to four digits. This is simply because the states
with very large wealth have low probability. However, this stabilization
procedure produced smoother heat maps for large wealth values, without
altering the summary statistics appreciably.

\subsection{Convergence Test: Synthetic Market}\label{app_convergence}
Table~\ref{conservative_accuracy} shows a detailed convergence test
for the base case problem given in Table \ref{base_case_1}, for the
EW-ES problem. The results are given for a sequence of grid sizes,
for the dynamic programming algorithm in \citep{forsyth:2022}
and Appendix~\ref{Numerical_Appendix}. The dynamic programming algorithm
appears to converge at roughly a second order rate. The optimal control
computed using dynamic programming is stored, and then used in Monte
Carlo computations. The Monte Carlo results are in good agreement with
the dynamic programming solution. For all the numerical examples, we
will use the $2048 \times 2048$ grid, since this seems to be accurate
enough for our purposes.

\begin{table}[hbt!]
\begin{center}
\begin{tabular}{lccc|cc} \toprule
 & \multicolumn{3}{|c|}{Algorithm in  \citep{forsyth:2022}
                        and Appendix \ref{Numerical_Appendix}} 
 & \multicolumn{2}{c}{Monte Carlo} \\ \midrule
Grid & LS & $E[\sum_i \qq_i]/M$ & Value Function &  LS
 & $E[\sum_i \qq_i]/M$ \\ \midrule
$512 \times 512$   & -1.40884 & 50.9082 & 1484.981 & -1.26443 & 50.938 \\
$1024 \times 1024$ & -1.32050 & 50.9491  & 1488.864 & -1.27396 & 50.953 \\ 
$2048 \times 2048$ & -1.30148 & 50.9643  & 1489.880 & -1.28189 & 50.963 \\
\bottomrule
\end{tabular}
\caption{
EW-LS convergence test.
Real stock index: deflated real capitalization
weighted CRSP, real bond index: deflated 30 day T-bills. Scenario in
Table~\ref{base_case_1}. Parameters in Table~\ref{fit_params}. The Monte
Carlo method used $2.56 \times 10^6$ simulations. The MC method used the
control from the solving the PIDEs as described in Appendix \ref{Numerical_Appendix}. $\kappa = 30,
\XL = 0.0$. Grid refers to the grid used in the Algorithm in Appendix
\ref{Numerical_Appendix}: $n_x \times n_b$, where $n_x$ is the number of
nodes in the $\log s$ direction, and $n_b$ is the number of nodes in the
$\log b$ direction. Units: thousands of dollars (real). $M$ is the total
number of withdrawals (rebalancing dates).
}
\label{conservative_accuracy}
\end{center}
\end{table}

\section{Effect of Stabilization term} \label{Appendix_stablizaton}
Recall that the optimization problem, for all objective functions,
becomes ill-posed along any path where $W_t \gg \XL$, $t \rightarrow T$.
To remove this problem, the stabilization term
\begin{eqnarray}
   \epsilon E[ W_T] 
\end{eqnarray}
is added to each objective function.  We set $|\epsilon| \ll 1$, to ensure
that this term has little effect unless we are in the ill-posed region.
Essentially, if $\epsilon < 0$, then this forces the portfolio to invest
100\% in bonds.  On the other hand, if $\epsilon > 0$, then the portfolio
will invest 100\% in stocks.  We remark here that these choices are essentially
arbitrary: by assumption, the 95-year old retiree
has a short life expectancy, and has large wealth, so that even with maximum
withdrawals, there is almost zero probability of running out of cash.

To verify that the choice of positive or negative $\epsilon$ has little effect
near $W_t = \XL$,  Figure \ref{stabilize_PS_fig} shows the CDF curves
for the EW-PS strategy (EW = 53.0), for both positive and negative $\epsilon$.
We can see that both curves overlap for $W_T \leq 100$.  Consequently, left tail risk
measures will be identical for both cases, and there will be
no differences in average withdrawals, since we will be constrained
by the maximum withdrawal specification.  To the right
of $W_T = 100$, we can see that for $\epsilon >0$, there is higher probability
of obtaining larger $W_T$ compared to the case $\epsilon <0$.  This is of
course expected, since investing in all stocks, (when $W_t$ is large) will have a larger expected
portfolio value.

The CDF curves for $\pm \epsilon$ for EW-ES and EW-LS policies are similar.

\begin{figure}[tb]
\centerline{%
\includegraphics[width=2.0in]{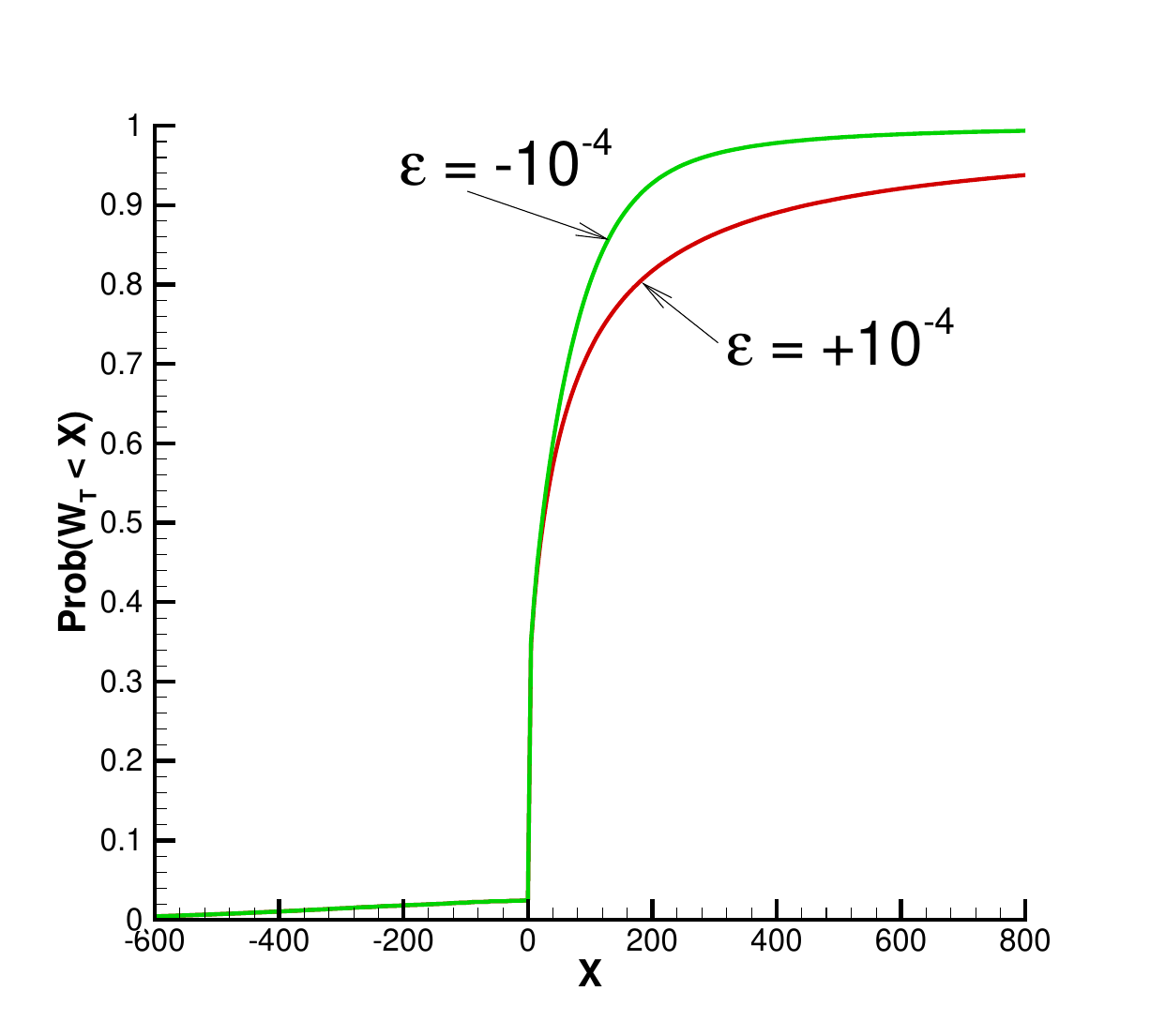}
}
\caption{ 
EW-PS CDF of the terminal wealth, with stabilization parameters shown.
Point on curve where $EW \simeq 53.0$.
Real stock index: deflated real capitalization
weighted CRSP, real bond index: deflated 30 day T-bills. Scenario in
Table~\ref{base_case_1}. Parameters in Table~\ref{fit_params}.
Synthetic market.
}

\label{stabilize_PS_fig}
\end{figure}

\begin{singlespace}
\bibliographystyle{chicago}
\bibliography{paper}
\end{singlespace}

\end{document}